\documentclass[a4paper, 12pt]{amsart}
\usepackage{amssymb, mathdots}
\usepackage[top=3cm, bottom=2cm, left=2cm, right=2cm]{geometry}

\newtheorem{thm}{Theorem}[section]
\newtheorem*{thm*}{Theorem}
\newtheorem{lem}[thm]{Lemma}
\newtheorem{prp}[thm]{Proposition}
\newtheorem{cor}[thm]{Corollary}

\theoremstyle{remark}
\newtheorem{rmk}[thm]{Remark}

\theoremstyle{definition}
\newtheorem{dfn}[thm]{Definition}
\newtheorem{ntn}[thm]{Notation}

\numberwithin{equation}{section}

\newcommand{\CC}{\mathbb{C}}
\newcommand{\NN}{\mathbb{N}}

\newcommand{\RR}{\mathbb{R}}

\newcommand{\ZZ}{\mathbb{Z}}

\newcommand{\Bb}{\mathcal{B}}

\newcommand{\Kk}{\mathcal{K}}
\newcommand{\Tt}{\mathcal{T}}
\newcommand{\Oo}{\mathcal{O}}

\newcommand{\ndim}{\operatorname{dim_{\operatorname{nuc}}}}
\newcommand{\id}{\operatorname{id}}

\newcommand{\MCE}{\operatorname{MCE}}
\newcommand{\lsp}{\operatorname{span}}

\newcommand{\ftn}[3]{#1  : #2 \rightarrow #3}
\newcommand{\setof}[2]{\left\{#1 \mid #2 \right\}}
\newcommand{\Mor}{\operatorname{Mor}}
\newcommand{\Obj}{\operatorname{Obj}}

\title{UCT-Kirchberg algebras have nuclear dimension one}

\author{Efren Ruiz}
\address{Department of Mathematics\\University of Hawaii,
Hilo\\200 W. Kawili St.\\
Hilo, Hawaii\\
96720-4091 USA}
\email{ruize@hawaii.edu}

\author{Aidan Sims}
\address{School of Mathematics and Applied Statistics \\
Faculty of Engineering and Information Sciences \\
University of Wollongong NSW 2522\\
Australia}
\email{asims@uow.edu.au}

\author{Adam P. W. S{\o}rensen}
\address{School of Mathematics and Applied Statistics \\
Faculty of Engineering and Information Sciences \\
University of Wollongong NSW 2522}
\address{Department of Mathematical Sciences \\
University of Copenhagen \\
Universitetsparken 5, 2100 Copenhagen \O \\
Denmark}
\email{apws@math.uio.no}

\date{\today}

\thanks{This research has been supported by funding from the Simons Foundation
(Collaboration Grant \#279369 to Efren Ruiz), the Australian Research Council,
and the Danish Council for Independent Research {\textbar} Natural Sciences.
The first named author would like to thank the School of Mathematics and Applied
Statistics at the University of Wollongong for hospitality during his visit
where this work was carried out.}

\begin{document}

\begin{abstract}
We prove that every Kirchberg algebra in the UCT class has nuclear dimension~1. We first
show that Kirchberg $2$-graph algebras with trivial $K_0$ and finite $K_1$ have nuclear
dimension~1 by adapting a technique developed by Winter and Zacharias for Cuntz algebras.
We then prove that every Kirchberg algebra in the UCT class is a direct limit of
$2$-graph algebras to obtain our main theorem.
\end{abstract}

\maketitle

\section{Introduction}

Nuclear dimension for $C^*$-algebras, introduced by Winter and Zacharias in
\cite{WinterZacharias:Adv10}, is a noncommutative notion of rank based on covering
dimension for topological spaces. It has been shown \cite{SatoWhiteWinter:xx14,
TomsWinter:GAFA13, Winter:IM12} to be closely related to $\mathcal{Z}$-stability and
hence to the classification program for simple nuclear $C^*$-algebras. Winter and
Zacharias showed that all UCT-Kirchberg algebras (i.e., separable, nuclear, simple,
purely infinite $C^{*}$-algebras in the UCT class) have nuclear dimension at most 5 and
asked whether the precise value of their dimension is determined by algebraic properties
of their $K$-groups, such as torsion \cite[Problem~9.2]{WinterZacharias:Adv10}. Matui and
Sato \cite{MatuiSato} subsequently improved the estimate for simple Kirchberg algebras
from~5 to~3, and their result is valid for non-UCT Kirchberg algebras, if any exist; and
Barlak, Enders, Matui, Szab\'o and Winter have showed how to recover the Matui-Sato
estimate from a general relationship between the nuclear dimension of an
$\Oo_\infty$-stable $C^*$-algebra and its $\Oo_2$-stablization that implies, in
particular, that every $\Oo_\infty$-absorbing $C^*$-algebra with compact metrizable
primitive-ideal space has nuclear dimension at most~7
\cite{BarlakEndersEtAl:arXiv:1312.6289}. For UCT-Kirchberg algebras, a further
improvement due to Enders \cite{Enders:arXiv:1405.6538} shows that every UCT-Kirchberg
algebra with torsion-free $K_1$ has nuclear dimension~1. But the question remained open
whether torsion in $K_1$ precludes having nuclear dimension~1. Here we prove that every
UCT-Kirchberg algebra, regardless of its K-theory, has nuclear dimension~1. This settles
the question posed by Winter and Zacharias, and suggests another question: does every
simple separable non-AF $C^*$-algebra with finite nuclear dimension have nuclear
dimension~1? Indeed, the appearance of the present paper in preprint form motivated
Bosa--Brown--Sato--Tikuisis--White--Winter in their remarkable recent preprint
\cite{BBSTWW} to pursue the optimal bound of~1 for the nuclear dimension of very broad
classes of simple $C^*$-algebras (see \cite[page~4]{BBSTWW}).

We recall the definition of nuclear dimension. A completely positive map $\phi$ between
$C^*$-algebras is order zero if $ab = 0$ implies $\phi(a)\phi(b) = 0$ for positive $a,b$.
A separable $C^{*}$-algebra $A$ has \emph{nuclear dimension $r$}, denoted by
$\mathrm{dim}_{\mathrm{nuc}} ( A ) = r$, if $r$ is the least element in $\NN \cup \{
\infty \}$ for which there is a sequence of order-$r$ factorisations $\phi_n \circ
\psi_n$ through finite dimensional $C^*$-algebras $\mathcal{F}_n$ that pointwise
approximate the identity map on $A$. That is, there exist finite dimensional
$C^{*}$-algebras $( \mathcal{F}_{n} )_{ n \in \NN }$, completely positive, contractive
linear maps $( \ftn{\psi_{n}}{A}{ \mathcal{F}_{n} } )_{n \in \NN }$, and completely
positive linear maps $( \ftn{\phi_{n}}{ \mathcal{F}_{n} }{A} )_{n \in \NN }$ such that
\begin{enumerate}
\item $\lim_{n \to \infty } \| a - \phi_{n} \circ \psi_{n} (a) \| = 0$ for all $a \in
    A$ and
\item each $\mathcal{F}_{n}$ has a decomposition $\bigoplus_{i=0}^{r}
    \mathcal{F}_{n,i}$ such that $\phi_{n} \vert_{ \mathcal{F}_{n,i}}$ is an
    order-zero completely positive contraction for each $i$.
\end{enumerate}
Winter and Zacharias' calculation of nuclear dimension for Cuntz algebras in
\cite{WinterZacharias:Adv10} is related to a construction of Kribs and Solel
\cite{KribsSolel:JAMS07} which builds from a directed graph $E$ a sequence of directed
graphs $( E(n) )_{ n = 1}^{\infty}$ comprising a kind of generalised combinatorial
solenoid. The first two authors, with Tomforde, used the Kribs-Solel construction
explicitly to compute nuclear dimension of many purely infinite nonsimple graph algebras
in \cite{RuizSimsTomforde:arXiv:1312.0507}. The key feature of $E(n)$ used in
nuclear-dimension calculations is that there are inclusions of the Toeplitz algebras
$\ftn{\iota_n}{\Tt C^*(E)}{\Tt C^*(E(n))}$ that can be approximated modulo compacts by
order-1 factorisations through finite-dimensional $C^*$-algebras. These are parlayed into
an approximation of the identity on $C^*(E)$ using a completely positive splitting
$C^*(E) \to \Tt C^*(E)$ (which exists since every graph algebra is nuclear), a suitable
sequence of homomorphisms $j_n : C^*(E(n)) \to C^*(E) \otimes \Kk$, and classification
results for purely infinite $C^*$-algebras.

Here, we develop a version of this machinery for higher-rank graphs and their
$C^*$-algebras as introduced in \cite{KumjianPask:NYJM00}.  We use this, and the fact
that each Kirchberg algebra with trivial $K_0$ and finite $K_1$ is isomorphic to a tensor
product of purely infinite simple graph $C^*$-algebras, to show that such Kirchberg
algebras have nuclear dimension~1. This is, on the face of it, somewhat surprising: one
expects the tensor product of $C^*$-algebras of nuclear dimension $n$ and $m$ to have
nuclear dimension $(n+1)(m+1)-1$, as is the case for commutative $C^*$-algebras. Here the
underlying combinatorial model plays a key role. For graph $C^*$-algebras
\cite{RuizSimsTomforde:arXiv:1312.0507}, following
\cite[Section~7]{WinterZacharias:Adv10}, the $n$\textsuperscript{th} order-1
factorisation through finite-dimensional $C^*$-algebras was obtained by producing a pair
of pavings of the combinatorial quadrant $\NN \times \NN$ by copies of a completely
positive contractive $n \times n$ real matrix so that when the two pavings are
superimposed, entries at any fixed distance from the diagonal, and sufficiently far from
the origin, approach~1 as $n \to \infty$. For $2$-graphs, we must perform an analogous
decomposition in $\NN^2 \times \NN^2$, and the expected increase in dimension turns out
to be illusory: facing no topological constraints, we can just pick convenient bijections
of $\NN^2 \times \NN^2$ onto $\NN \times \NN$ that carry points close to the diagonal to
points close to the diagonal. In essence, this is the idea underlying the technical
argument of Theorem~\ref{thm: order 1 approximation}.

To deduce our main theorem, we apply the Kirchberg-Phillips theorem to show that every
stable UCT-Kirchberg algebra is a direct limit of purely infinite $2$-graph
$C^*$-algebras that are known to have nuclear dimension~1 either by Enders' results or by
the result of the preceding paragraph. We believe that this realisation of Kirchberg
algebras using $2$-graph $C^*$-algebras has independent interest: many important
$C^*$-algebraic properties are preserved under direct limits, and as our approach here
shows, the combinatorial structure of $2$-graphs can provide a good line of attack in
establishing structural properties of the associated $C^*$-algebras. Related
combinatorial models have been used to great effect to study weak semiprojectivity
\cite{Spielberg:CMB07} and prime-order automorphisms \cite{Spielberg:JraM07} of
UCT-Kirchberg algebras. Our techniques also have potential applications to calculations
of nuclear dimension for large classes of nonsimple Kirchberg algebras along the lines of
\cite{RuizSimsTomforde:arXiv:1312.0507} (see Remark~\ref{rmk:SpecSeq}).

\smallskip

We start with some background on higher-rank graphs in Section~\ref{sec:k-graphs and
repns}. In Section~\ref{sec:k-graphKS}, we show how to generalise the Kribs-Solel
construction to higher-rank graphs, and produce analogues of the homomorphisms $\iota_n$
and $j_n$ discussed in the preceding paragraph. In Section~\ref{sec:CP1-graphKS}, we
investigate how our construction behaves with respect to the cartesian-product
construction for $k$-graphs \cite{KumjianPask:NYJM00}; this allows us to relate the
results of the preceding section to tensor products of graph $C^*$-algebras. In
Section~\ref{sec:order1factorisations}, we show that for $2$-graphs, the maps
$\ftn{\tilde\iota_n}{C^*( \Lambda )}{C^*( \Lambda(n))}$ induced by the $\iota_n$ can be
asymptotically approximated by sums of two order-zero maps through AF-algebras. In
Section~\ref{sec:main}, we prove our main result. We first show that if $E$ and $F$ are
1-graphs whose $C^*$-algebras are Kirchberg algebras with $K$-theory $(T, 0)$ and
$(0,\ZZ)$ respectively, where $T$ is a finite abelian group, then for the $2$-graph
$\Lambda = E \times F$, the composition $j_{(n_{1} , n_{2} )} \circ \tilde\iota_{(n_{1},
n_{2} )}$ implements multiplication by $n_1n_2$ in $K_*(C^*(\Lambda)) \cong (0, T)$. By
choosing increasing $(n_{1}, n_{2} )$ for which multiplication by $n_1n_2$ is the
identity on $T$, and applying classification machinery, we deduce that UCT-Kirchberg
algebras with trivial $K_0$ and finite $K_1$ have nuclear dimension~1. We then prove our
main result by combining this with Enders' results and a direct-limit argument.
%
%We finish with an appendix in which we provide a second and more general proof that the
%maps $j_{(n_{1}, n_{2} )} \circ \tilde\iota_{(n_{1}, n_{2})}$ induce multiplication by
%$n_1n_2$ in $K$-theory for $2$-graph algebras. Combined with the identity $n^2 -
%(n-1)(n+1) = 1$ and the argument of
%\cite[Proposition~4.5]{RuizSimsTomforde:arXiv:1312.0507}, this could be used to obtain a
%direct proof that Kirchberg $2$-graph algebras have nuclear dimension~1. However the
%argument of Appendix~\ref{app:K-theory map} requires naturality of Kasparov's spectral
%sequence, for which no explicit proof appears to have been published. So we set this
%material aside from the main body of the paper.

\section{Higher rank graphs and their $C^{*}$-algebras} \label{sec:k-graphs and repns}

We recall the standard conventions for $k$-graphs and their $C^*$-algebras introduced in
\cite{KumjianPask:NYJM00}. We regard $\NN^k$ as an additive semigroup with identity 0.
For $m,n \in \NN^k$, we write $m \vee n$ for their coordinatewise maximum. We write $n
\leq m$ if $n_{i} \leq m_{i}$ for all $i$. We also write $n < m$ to mean $n_{i} < m_{i}$
for all $i$. Warning: this convention means that $n < m$ and $n \lneq m$ mean different
things.

\begin{dfn}[{See~\cite{KumjianPask:NYJM00}}] \label{dfn:k-graph}
Let $k \in \NN \setminus \{0\}$.  A \emph{graph of rank $k$}, or \emph{$k$-graph}, is a
pair $(\Lambda,d)$ where $\Lambda$ is a countable category and $d$ is a functor from
$\Lambda$ to $\NN^k$ that satisfies the \emph{factorisation property}:  for all $\lambda
\in \Mor(\Lambda)$ and all $m,n \in \NN^k$ such that $d(\lambda) = m+n$, there exist
unique morphisms $\mu$ and $\nu$ in $\Mor(\Lambda)$ such that $d(\mu) = m$, $d(\nu) = n$
and $\lambda = \mu\nu$.
\end{dfn}

Since we are regarding $k$-graphs as generalised directed graphs, we refer to elements of
$\Mor(\Lambda)$ as \emph{paths}.  The factorisation property implies that $\setof{\id_o}{
o \in \Obj(\Lambda) } = \setof{\lambda \in \Mor(\Lambda)}{ d(\lambda) = 0}$. So the
codomain and domain maps $\ftn{\operatorname{cod},
\operatorname{dom}}{\Mor(\Lambda)}{\Obj(\Lambda)}$ determine maps $r : \lambda \mapsto
\id_{\operatorname{cod}(\lambda)}$ and $s : \lambda \mapsto
\id_{\operatorname{dom}(\lambda)}$ from $\Mor(\Lambda)$ to $d^{-1}(0)$. We refer to the
elements of $d^{-1}(0)$ as \emph{vertices}, and call $r(\lambda)$ and $s(\lambda)$ the
range and source of $\lambda$. We have $r(\lambda)\lambda = \lambda = \lambda
s(\lambda)$. We write $\lambda \in \Lambda$ to mean $\lambda \in \Mor(\Lambda)$.

We use the following notation from \cite{PaskQuiggEtAl:JA05}: given $\lambda \in \Lambda$
and $E \subseteq \Lambda$, we define
\[
\lambda E :=\setof{\lambda\mu}{\mu \in E, r(\mu) = s(\lambda)} \  \text{ and } \ E\lambda := \setof{\mu\lambda }{\mu \in E, s(\mu) =
r(\lambda) }.
\]
In particular if $d(v) = 0$, then $vE = \setof{\lambda \in E}{ r(\lambda) = v }$, and $Ev
= \setof{\lambda \in E}{ s(\lambda) = v }$.

For $n \in \NN^{k}$, we let $\Lambda^{n} = d^{-1} (n)$.  For $n < m$, we set $\Lambda^{[n
, m)} = \setof{\lambda \in \Lambda}{n \leq d(\mu) < m}$.  We use the convention that for
$m \le n \le d(\lambda)$, the path $\lambda(m,n)$ is the unique element of
$\Lambda^{n-m}$ such that $\lambda = \lambda'\lambda(m,n)\lambda''$ for some $\lambda'
\in \Lambda^m$. An application of the factorisation property shows that for $m \le
d(\lambda)$ we have $\lambda = \lambda(0,m)\lambda(m,d(\lambda))$.

As in \cite{KumjianPask:NYJM00}, we say that $\Lambda$ is \emph{row-finite} if $v
\Lambda^n$ is finite for each $v \in \Lambda^0$ and $n \in \NN^k$.  We say that $\Lambda$
has \emph{no sources} if each $v \Lambda^n \not= \emptyset$. All $k$-graphs in this paper
will be row-finite with no sources.

\begin{dfn}[See~{\cite{RaeburnSimsEtAl:JFA04}}] \label{dfn:common extensions}
Let $(\Lambda,d)$ be a $k$-graph. Given $\mu,\nu \in \Lambda$, we say that $\lambda$ is a
\emph{minimal common extension} of $\mu$ and $\nu$ if $\lambda \in \mu\Lambda \cap
\nu\Lambda$ and $d(\lambda) = d(\mu) \vee d(\nu)$. We denote the collection $\mu\Lambda
\cap \nu\Lambda \cap \Lambda^{d(\mu) \vee d(\nu)}$ of all minimal common extensions of
$\mu$ and $\nu$ by $\MCE(\mu,\nu)$. We define
\[\Lambda^{\mathrm{min}}(\mu,\nu) := \setof{(\alpha,\beta) \in \Lambda \times \Lambda }{ \mu\alpha =  \nu\beta \in \MCE(\mu,\nu) }.
\]
\end{dfn}
For a row-finite $k$-graph $\Lambda$, the set $\MCE(\mu,\nu)$ is finite for all $\mu, \nu \in \Lambda$, since each $\MCE(\mu,\nu) \subseteq r(\mu)\Lambda^{d(\mu) \vee d(\nu)}$.  The factorisation property
ensures that $(\alpha,\beta) \mapsto \mu\alpha$ is a bijection from
$\Lambda^{\min}(\mu,\nu)$ to $\MCE(\mu,\nu)$.

The following definition of a Toeplitz-Cuntz-Krieger family for a higher-rank graph is
essentially \cite[Definition~7.1]{RaeburnSims:JOT05}, with the appropriate changes of
conventions to translate from product-systems of graphs to $k$-graphs (see also
\cite[Section~2.2]{anHuefLacaEtAl:JFA14}).

\begin{dfn}
Let $\Lambda$ be a row-finite $k$-graph with no sources.  A \emph{Toeplitz-Cuntz-Krieger
$\Lambda$-family} is a collection $\{t_{\lambda} \}_{\lambda \in \Lambda}$ of partial
isometries in a $C^{*}$-algebra satisfying
\begin{itemize}
\item[(TCK1)] $\{t_{v}\}_{v \in \Lambda^{0}}$ is a collection of mutually
    orthogonal projections;
\item[(TCK2)] $t_{\lambda} t_{\mu} = \delta_{s(\lambda), r(\mu)} t_{\lambda \mu}$ for
    all $\lambda, \mu \in \Lambda$;
\item[(TCK3)] $t_{\lambda}^{*} t_{\lambda} = t_{s(\lambda)}$ for all $\lambda \in
    \Lambda$; and
\item[(TCK4)] $t_{\lambda}^{*} t_{\mu} =  \sum_{(\alpha, \beta) \in
    \Lambda^{\mathrm{min}} (\lambda, \mu)} t_{\alpha} t_{\beta}^{*}$ for all
    $\lambda, \mu \in \Lambda$.
\end{itemize}

\begin{rmk}
In previous treatments (see, for example, \cite{anHuefLacaEtAl:JFA14}), the definition of
a Toeplitz-Cuntz-Krieger $\Lambda$-family for a row-finite $k$-graph $\Lambda$ with no
sources has included the additional relation
\begin{equation}\label{eq:TCK5}
\sum_{\lambda \in v\Lambda^n} t_\lambda t^*_\lambda \le t_v\qquad\text{ for all $v \in \Lambda^0$ and $n \in \NN^k$.}
\end{equation}
But in fact this relation follows from the other four. To see this, observe that if
$\mu,\nu \in \Lambda^n$, then $\Lambda^{\mathrm{min}}(\mu,\nu) = \{s(\mu),s(\mu)\}$ if
$\mu = \nu$ and is empty otherwise. So~(TCK4) shows that the $t_\lambda t^*_\lambda$ for
$\lambda \in \Lambda^n$ are mutually orthogonal, and so the sum on the left-hand side
of~\eqref{eq:TCK5} is a projection. For each $\Lambda$, we have
$\Lambda^{\mathrm{min}}(\lambda, r(\lambda)) = \{(s(\lambda), \lambda)\}$. Hence
\[
\Big(\sum_{\lambda \in v\Lambda^n} t_\lambda t^*_\lambda\Big) t_v
    = \sum_{\lambda \in v\Lambda^n} \sum_{(\alpha,\beta) \in \Lambda^{\mathrm{min}}(\lambda, r(\lambda))} t_\lambda t_\alpha t^*_\beta
    = \sum_{\lambda \in v\Lambda^n} t_\lambda t_{s(\lambda)} t^*_\lambda
    = \sum_{\lambda \in v\Lambda^n} t_\lambda t^*_\lambda.
\]
\end{rmk}

As in \cite{KumjianPask:NYJM00}, a \emph{Cuntz-Krieger $\Lambda$-family} is a collection
$\{s_{\lambda}\}_{\lambda \in \Lambda}$ of partial isometries in a $C^{*}$-algebra
satisfying (TCK1), (TCK2), (TCK3), and
\begin{itemize}
\item[(CK)] $s_{v} = \sum_{\lambda \in v \Lambda^{n}} s_{\lambda} s_{\lambda}^{*}$
    for each $v \in \Lambda^{0}$ and $n \in \NN^{k}$.
\end{itemize}

Let $\Lambda$ be a row-finite $k$-graph with no sources.  There is a universal
$C^{*}$-algebra $\Tt C^{*} (\Lambda)$ generated by a universal Toeplitz-Cuntz-Krieger
$\Lambda$-family $\{t_{\lambda}\}_{\lambda \in \Lambda}$. We call this $C^{*}$-algebra
the \emph{Toeplitz algebra of $\Lambda$}.  There is also a universal $C^{*}$-algebra
$C^{*} (\Lambda)$ generated by a universal Cuntz-Krieger $\Lambda$-family
$\{s_{\lambda}\}_{\lambda \in \Lambda}$.  We call this $C^{*}$-algebra the
\emph{Cuntz-Krieger algebra of $\Lambda$}, or just the $C^*$-algebra of $\Lambda$.
\end{dfn}

\section{The Kribs-Solel construction for $k$-graphs}\label{sec:k-graphKS}

For the duration of this section, we fix a row-finite $k$-graph $\Lambda$ with no
sources. The key tool for understanding nuclear dimension of graph algebras in
\cite{RuizSimsTomforde:arXiv:1312.0507} was a construction due to Kribs and Solel
\cite{KribsSolel:JAMS07}. The first step in our analysis here is to adapt this
construction to $k$-graphs.

Choose $n = (n_1, \dots, n_k) \in \NN^k$ with each $n_i \ge 1$. Let
\[
H_{n} := \setof{a n}{ a \in \ZZ^k } = \setof{m \in \ZZ^k}{ m_i/n_i \in \ZZ \text{ for all $i$} }.
\]
We will often just write $H$ for $H_{n}$. For $m \in \NN^k$, we write $[m]$ for $m + H
\in \ZZ^k/H$. We often identify $\ZZ^k/H$ as a set with $\setof{m \in \NN^k}{ m < n }$.

For $\lambda \in \Lambda$, we define
\[[\lambda]_H := \lambda(0, [d(\lambda)]),
\]
and we usually write $[\lambda]$ for $[\lambda]_H$. So $[\lambda]$ is the unique element
of $\Lambda$ such that $d([\lambda]) < n$ and $\lambda = [\lambda]\lambda'$ with
$d(\lambda') \in H$. The factorisation property implies that if $d(\mu) \in H$, then
$[\lambda\mu] = [\lambda]$.

Following \cite{RuizSimsTomforde:arXiv:1312.0507}, we write $\Lambda^{<n} := \setof{\lambda
\in \Lambda }{d(\lambda) <n }$.  Let
\[\Lambda(n) := \setof{(\lambda, \lambda') \in \Lambda \times \Lambda^{<n} }{s(\lambda) = r(\lambda')}.
\]
We aim to make this set into a $k$-graph. For $(\lambda, \lambda') \in \Lambda(n)$,
define
\[d((\lambda, \lambda')) := d(\lambda).
\]
So $\Lambda(n)^0 = \setof{(r(\lambda),\lambda) }{ \lambda \in \Lambda^{<n} }$. Define
$\ftn{r,s}{\Lambda(n)}{\Lambda(n)^0}$ by
\begin{align*}
s((\lambda, \lambda')) &:= (s(\lambda), \lambda'), \quad\text{and}\\
r((\lambda, \lambda')) &:= (r(\lambda), [\lambda\lambda']).
\end{align*}
Identify $\Lambda(n)^0$ with $\Lambda^{<n}$ via $(r(\lambda),\lambda) \mapsto \lambda$.
Then $s((\lambda,\lambda')) = \lambda'$ and $r((\lambda, \lambda')) = [\lambda\lambda']$.
Suppose that $s((\lambda,\lambda')) = r((\mu,\mu'))$; that is, $\lambda' = [\mu\mu']$.
Then we define
\[
(\lambda,\lambda')(\mu,\mu') := \big(\lambda\mu,\mu'\big).
\]

\begin{lem}
Under the operations just described, $\Lambda(n)$ is a row-finite $k$-graph with no
sources.
\end{lem}
\begin{proof}
We show that $\Lambda(n)$ is a category. We first check that $s$ and $r$ are compatible
with composition. Suppose that $s((\lambda,\lambda')) = r((\mu,\mu'))$. Then
\[s((\lambda,\lambda')(\mu,\mu'))
    = s((\lambda\mu, \mu'))
    = \mu' = s((\mu,\mu')).
\]
Writing $\mu\mu' = [\mu\mu']\tau = \lambda'\tau$, we have
\[r\big((\lambda,\lambda')(\mu,\mu')\big)
    = r((\lambda\mu, \mu'))
    = \big[\lambda\mu\mu'\big]
    = \big[\lambda[\mu\mu']\tau\big]
    = [\lambda\lambda'\tau];
\]
and since $d(\tau) \in H$, we have $r((\lambda,\lambda')(\mu,\mu')) =
[\lambda\lambda'\tau] = [\lambda\lambda'] = r((\lambda,\lambda'))$.

We now check that $r((\lambda,\lambda'))$ and $s((\lambda,\lambda'))$ act as left- and
right identities for $(\lambda, \lambda')$:
\[r((\lambda,\lambda'))(\lambda,\lambda')
    = (r(\lambda), [\lambda\lambda'])(\lambda,\lambda')
    = (r(\lambda)\lambda, \lambda')
    = (\lambda,\lambda'),
\]
and
\[(\lambda,\lambda')s((\lambda,\lambda'))
    = (\lambda, \lambda')\big(s(\lambda), \lambda'\big)
    = (\lambda s(\lambda), \lambda')
    = (\lambda, \lambda').
\]
To check associativity, suppose that $s((\lambda,\lambda')) = r((\mu,\mu'))$ and
$s((\mu,\mu')) = r((\nu,\nu'))$.  Then
\begin{align*}
((\lambda, \lambda')(\mu,\mu'))(\nu,\nu')
    &= (\lambda\mu, \mu')(\nu, \nu')
    = (\lambda\mu\nu, \nu'),\\
    &= (\lambda,\lambda')(\mu\nu, \nu')
    = (\lambda,\lambda')((\mu,\mu')(\nu,\nu')).
\end{align*}
So $\Lambda(n)$ is a category.

We check that $d$ is a functor:
\[d((\lambda,\lambda')(\mu,\mu'))
    = d((\lambda\mu, \mu'))
    = d(\lambda\mu) = d(\lambda) + d(\mu)
    = d((\lambda,\lambda')) + d((\mu, \mu')).
\]
Now we check the factorisation property. Suppose that $d((\lambda,\lambda')) = p+q$. Then
$d(\lambda) = p+q$, and the factorisation property in $\Lambda$ gives $\mu \in \Lambda^p$
and $\nu \in \Lambda^q$ such that $\lambda = \mu\nu$. Now $(\nu,\lambda') \in
\Lambda(n)^q$ and has range $r((\nu,\lambda')) = [\nu\lambda']$. Hence $(\mu, [\nu\lambda'])
\in \Lambda(n)^p r((\nu,\lambda'))$, and $(\mu,[\nu\lambda'])(\nu,\lambda') = (\mu\nu,
\lambda') = (\lambda, \lambda')$. For uniqueness, suppose that $(\alpha, \alpha') \in
\Lambda(n)^p$ and $(\beta,\beta') \in \Lambda(n)^q$ satisfy
$(\alpha,\alpha')(\beta,\beta') = (\lambda, \lambda')$. By definition of composition, we
have $(\alpha\beta, \beta') = (\lambda,\lambda')$. This forces $\beta' = \lambda'$ and
$\alpha\beta = \lambda$. Since $d(\alpha) = d((\alpha, \alpha')) = p$ and $d(\beta) =
d((\beta,\beta')) = q$, the factorisation property in $\Lambda$ forces $\alpha = \mu$ and
$\beta = \nu$. Since $(\alpha, \alpha')$ and $(\beta, \beta')$ are composable, we have
$\alpha' = s((\mu, \alpha')) = r ((\nu, \lambda')) = [\nu\lambda']$. Hence $\Lambda(n)$
is a $k$-graph.

To see that $\Lambda(n)$ is row-finite with no sources, take $(r(\lambda), \lambda) \in
\Lambda(n)^0$ and $m \in \NN^k$. Then
\begin{align*}
(r(\lambda),\lambda)\Lambda(n)^m
    &= \setof{(\mu, \mu')}{ \mu \in \Lambda^m, \mu' \in s(\mu)\Lambda^{<n}, [\mu\mu'] = \lambda } \\
    &= \setof{(\mu, \mu')}{ \mu \in \Lambda^m, \mu' \in s(\mu)\Lambda^{[d(\lambda) - m]}, [\mu\mu'] = \lambda }.
\end{align*}
Since $m + [d(\lambda) - m]$ is positive and congruent to $d(\lambda)$ mod $H$, we have
$m + [d(\lambda) - m] \ge d(\lambda)$. Let $p := m + [d(\lambda) - m] - d(\lambda) \in
H$. Then
\[(r(\lambda),\lambda)\Lambda(n)^m
    = \big\{\big((\lambda\nu)(0,m), (\lambda\nu)(m, p+d(\lambda))\big) \mathbin{\big|}  \nu \in s(\lambda)\Lambda^p\big\},
\]
which is finite and nonempty because $s(\lambda)\Lambda^p$ is finite and nonempty.
\end{proof}

To work with Toeplitz-Cuntz-Krieger $\Lambda(n)$-families we first compute
$\Lambda(n)^{\min}((\lambda,\lambda'),(\mu,\mu'))$.

\begin{lem}\label{lem:Lambda(n)min}
For $(\lambda,\lambda'), (\mu,\mu') \in \Lambda(n)$, if $[\lambda\lambda'] \not=
[\mu\mu']$, then $\Lambda(n)^{\min}((\lambda,\lambda'), (\mu,\mu')) = \emptyset$;
otherwise,
\begin{equation}\label{eq:Lmin desc}
\begin{split}
\Lambda(n)^{\min}((\lambda,\lambda'),(\mu,\mu'))
    = \big\{((\alpha,\tau), (\beta,\tau)) \mathbin{\big|} {}&(\alpha,\beta) \in \Lambda^{\min}(\lambda,\mu),\\
                &\quad\tau \in s(\alpha)\Lambda^{<n}, [\alpha\tau] = \lambda' \text{ and } [\beta\tau] = \mu'\big\}.
\end{split}
\end{equation}
\end{lem}
\begin{proof}
If $[\lambda\lambda'] \not= [\mu\mu']$, then $r((\lambda, \lambda') ) \not= r((\mu,
\mu'))$, and so $\Lambda(n)^{\min}((\lambda,\lambda'), (\mu,\mu')) = \emptyset$.

Suppose that $[\lambda\lambda'] = [\mu\mu']$.  Suppose further that $(\alpha,\beta) \in
\Lambda^{\min}(\lambda,\mu)$, that $\tau \in s(\alpha)\Lambda^{<n}$, and that
$[\alpha\tau] = \lambda'$ and $[\beta\tau] = \mu'$. Then $(\alpha, \tau), (\beta, \tau)
\in \Lambda(n)$, and $r((\alpha,\tau)) = \lambda' = s((\lambda,\lambda'))$ and
$r((\beta,\tau)) = \mu' = s((\mu,\mu'))$. We have
\[(\lambda,\lambda')(\alpha,\tau) = (\lambda\alpha, \tau)
    = (\mu\beta, \tau) = (\mu,\mu')(\beta,\tau).
\]
Since $d((\lambda\alpha,\tau)) = d(\lambda\alpha) = d(\lambda) \vee d(\mu) = d( ( \lambda, \lambda') ) \vee d( (\mu, \mu') )$, we have
$((\alpha,\tau), (\beta,\tau)) \in \Lambda(n)^{\min}((\lambda,\lambda'), (\mu,\mu'))$.

Conversely, suppose that $(\alpha, \tau), (\beta,\rho) \in
\Lambda(n)^{\min}((\lambda,\lambda'), (\mu,\mu'))$. Then
\begin{equation}\label{eq:CE calc}
(\lambda\alpha,\tau)
    = (\lambda,\lambda')(\alpha,\tau)
    = (\mu,\mu')(\beta,\rho)
    = (\mu\beta,\rho).
\end{equation}
So $\lambda\alpha = \mu\beta$, and
\[d(\lambda\alpha)
    = d((\lambda\alpha, \tau))
    = d((\lambda,\lambda')(\alpha,\tau))
    = d((\lambda,\lambda')) \vee d((\mu,\mu'))
    = d(\lambda) \vee d(\mu),
\]
so $(\alpha,\beta) \in \Lambda^{\min}(\lambda,\mu)$. By~\eqref{eq:CE calc}, $\tau =
s((\lambda\alpha, \tau)) = s((\mu\beta,\rho)) = \rho$. Since $[\alpha\tau] = r((\alpha,
\tau)) = s((\lambda,\lambda')) = \lambda'$ and $[\beta\tau] = r((\beta,\tau)) =
s((\mu,\mu')) = \mu'$, we deduce that $((\alpha,\tau), (\beta, \rho)) = ((\alpha, \tau),
(\beta, \tau))$ belongs to the right-hand side of~\eqref{eq:Lmin desc}.
\end{proof}

For each $n$ we now construct a homomorphism from $C^{*}(\Lambda)$ to $C^{*}(\Lambda(n))$
analogous to those for directed graphs described in
\cite[Lemma~2.5]{RuizSimsTomforde:arXiv:1312.0507}.

\begin{lem}\label{l:homomorphism KribSolel}
Let $\{t_{\lambda} \}_{\lambda \in \Lambda} \subseteq \Tt C^*(\Lambda)$ and
$\{t_{(\lambda, \lambda')} \}_{(\lambda, \lambda') \in \Lambda(n)} \subseteq \Tt
C^*(\Lambda(n))$ be the generating Toeplitz-Cuntz-Krieger families and let $\{s_{\lambda}
\}_{\lambda \in \Lambda} \subseteq C^*(\Lambda)$ and $\{s_{(\lambda, \lambda')}
\}_{(\lambda, \lambda') \in \Lambda(n)} \subseteq C^*(\Lambda(n))$ be the generating
Cuntz-Krieger families.  For $n \in \NN^{k}$, there are homomorphisms $\ftn{\iota_n}{\Tt
C^*(\Lambda)}{\Tt C^*(\Lambda(n))}$ and
$\ftn{\tilde\iota_n}{C^*(\Lambda)}{C^*(\Lambda(n))}$ such that
\begin{equation}\label{eq:iota formula}
\iota_n(t_\lambda) = \sum_{\lambda' \in s(\lambda)\Lambda^{<n}} t_{(\lambda, \lambda')} \quad\text{and}\quad
\tilde\iota_n(s_\lambda) = \sum_{\lambda' \in s(\lambda)\Lambda^{<n}} s_{(\lambda, \lambda')}.
\end{equation}
The homomorphism $\iota_{n}$ descends to the homomorphism $\tilde{\iota}_{n}$ under the
canonical quotient maps from Toeplitz algebras to Cuntz-Krieger algebras.
\end{lem}
\begin{proof}
For $\lambda \in \Lambda$, define $T_\lambda := \sum_{\lambda' \in
s(\lambda)\Lambda^{<n}} t_{(\lambda, \lambda')} \in \Tt C^*(\Lambda(n))$. We check that
$\{T_\lambda \}_{\lambda \in \Lambda}$ is a Toeplitz-Cuntz-Krieger $\Lambda$-family. Take
$v,w \in \Lambda^0$. Since $\{t_{(r(\nu),\nu)}\}_{ \nu \in\Lambda^{<n} }$ is a set of
mutually orthogonal projections,
\[T_v^* T_w
    = \sum_{\lambda \in v\Lambda^{<n}} t_{(v,\lambda)} \sum_{\mu \in w\Lambda^{<n}} t_{(w, \mu)}
    = \sum_{\lambda \in v\Lambda^{<n}, \mu \in w\Lambda^{<n}} \delta_{(v,\lambda), (w,\mu)} t_{( v, \lambda )}
    = \delta_{v,w} T_v,
\]
and so $\{T_v \}_{v \in \Lambda^{0}}$ are mutually orthogonal projections,
giving~(TCK1).

For $(\lambda,\lambda'), (\mu,\mu') \in \Lambda(n)$,
we have
\[t_{(\lambda,\lambda')}t_{(\mu,\mu')}
    =  \delta_{s((\lambda,\lambda')), r((\mu,\mu'))} t_{(\lambda,\lambda')(\mu,\mu')}
    = \delta_{s(\lambda), r(\mu)} \delta_{\lambda', [\mu\mu']} t_{(\lambda\mu, \mu')}.
\]
Hence, for $\lambda, \mu \in \Lambda$,
\begin{align*}
T_\lambda T_\mu
    &= \sum_{\lambda' \in s(\lambda)\Lambda^{<n}, \mu' \in s(\mu)\Lambda^{<n}} t_{(\lambda,\lambda')}t_{(\mu,\mu')} \\
    &= \sum_{\lambda' \in s(\lambda)\Lambda^{<n}, \mu' \in s(\mu)\Lambda^{<n}} \delta_{s(\lambda), r(\mu)} \delta_{\lambda', [\mu\mu']} t_{(\lambda\mu,\mu')} \\
    &= \sum_{\mu' \in s(\mu)\Lambda^{<n}} \delta_{s(\lambda), r(\mu)} t_{(\lambda\mu,\mu')}
    = \delta_{s(\lambda), r(\mu)} T_{\lambda\mu}.
\end{align*}
So $\{ T_{\lambda} \}_{ \lambda \in \Lambda}$ satisfies~(TCK2).

For (TCK3)~and~(TCK4), fix $\lambda,\mu \in \Lambda$. We calculate:
\begin{align*}
T^*_\lambda T_\mu
    &= \sum_{\lambda' \in s(\lambda)\Lambda^{<n}, \mu' \in s(\mu)\Lambda^{<n}} t^*_{(\lambda,\lambda')} t_{(\mu,\mu')} \\
    &= \sum_{\substack{\lambda' \in s(\lambda)\Lambda^{<n} \\ \mu' \in s(\mu)\Lambda^{<n}}}
        \Big(\sum_{((\alpha,\alpha'), (\beta,\beta')) \in \Lambda(n)^{\min}((\lambda,\lambda'),(\mu,\mu'))} t_{(\alpha,\alpha')} t^*_{(\beta,\beta')}\Big)\\
    &= \sum_{\substack{\lambda' \in s(\lambda)\Lambda^{<n} \\ \mu' \in s(\mu)\Lambda^{[d(\lambda\lambda') - d(\mu)]}}}
        \sum_{\substack{(\alpha, \beta) \in \Lambda^{\min}(\lambda,\mu), \tau \in s(\alpha)\Lambda^{<n}, \\
                        [\alpha\tau] = \lambda', [\beta\tau] = \mu'}}
            t_{(\alpha,\tau)} t^*_{(\beta,\tau)} \quad\text{(by Lemma~\ref{lem:Lambda(n)min})}\\
    &= \sum_{(\alpha, \beta) \in \Lambda^{\min}(\lambda,\mu)} \Big(\sum_{\tau \in s(\alpha)\Lambda^{<n}}
            t_{(\alpha,\tau)} t^*_{(\beta,\tau)}\Big).
\end{align*}
If $(\alpha,\beta) \in \Lambda^{\min}(\lambda, \mu)$ and $\tau, \rho \in \Lambda^{<n}$,
then $t_{(\alpha,\tau)} t^*_{(\beta,\rho)} \not= 0$ forces $\tau = s((\alpha,\tau)) =
s((\beta,\rho)) = \rho$. So summing over two variables $\tau \in s(\alpha)\Lambda^{<n}$
and $\rho \in s(\alpha)\Lambda^{<n} = s(\beta)\Lambda^{<n}$ adds no new nonzero terms to
the final line of the preceding calculation. Hence
\[T^*_\lambda T_\mu
    = \sum_{(\alpha, \beta) \in \Lambda^{\min}(\lambda,\mu)} \Big(\sum_{\tau \in s(\alpha)\Lambda^{<n}, \rho \in s(\beta)\Lambda^{<n}}
            t_{(\alpha,\tau)} t^*_{(\beta,\rho)}\Big)
    = \sum_{(\alpha,\beta) \in \Lambda^{\min}(\lambda,\mu)} T_\alpha T^*_\beta.
\]
This gives~(TCK4); and~(TCK3) then follows from~(TCK1) because
$\Lambda^{\min}(\lambda,\lambda) = \{(s(\lambda), s(\lambda))\}$. Hence
$\{T_\lambda\}_{\lambda \in \Lambda}$ is a Toeplitz-Cuntz-Krieger $\Lambda$-family.

The universal property of $\Tt C^*(\Lambda)$ gives a homomorphism $\iota_n : \Tt
C^*(\Lambda) \to \Tt C^*(\Lambda(n))$ such that
\[
\iota_n(t_\lambda) = T_\lambda = \sum_{\lambda' \in s(\lambda)\Lambda^{<n}} t_{(\lambda, \lambda')}
\]
for all $\lambda$. To see that $\iota_{n}$ descends to the desired homomorphism
$\ftn{\tilde\iota_{n}}{C^*(\Lambda)}{C^*(\Lambda(n))}$, let $\ftn{q_n}{\Tt
C^*(\Lambda(n))}{C^*(\Lambda(n))}$ denote the quotient map. We check that the family
$S_\lambda := q_n(T_\lambda)$ satisfies~(CK). For $v \in \Lambda^0$ and $m \in \NN^k$,
\[\sum_{\lambda \in v\Lambda^m} S_\lambda S^*_\lambda
    = \sum_{\lambda \in v\Lambda^m} \sum_{\mu,\nu \in s(\lambda)\Lambda^{<n}} s_{(\lambda,\mu)}s^*_{(\lambda,\nu)}.
\]
As above, $s_{(\lambda,\mu)}s^*_{(\lambda,\nu)} \not= 0$ forces $s((\lambda,\mu)) =
s((\lambda,\nu))$, and so $\mu = \nu$. Using this at the first equality and relation~(CK)
in $C^*(\Lambda(n))$ at the second-last equality, we calculate:
\begin{align*}
\sum_{\lambda \in v\Lambda^m} S_\lambda S^*_\lambda
    &= \sum_{\lambda \in v\Lambda^m} \sum_{\lambda' \in s(\lambda)\Lambda^{<n}} s_{(\lambda,\lambda')}s^*_{(\lambda,\lambda')}
    = \sum_{(\lambda,\lambda') \in \Lambda(n)^m, r(\lambda) = v} s_{(\lambda,\lambda')}s^*_{(\lambda,\lambda')} \\
    &= \sum_{\alpha \in v\Lambda^{<n}} \sum_{(\lambda,\lambda') \in \Lambda(n)^m, [\lambda\lambda'] = \alpha} s_{(\lambda,\lambda')}s^*_{(\lambda,\lambda')}
    = \sum_{\alpha \in v\Lambda^{<n}} \sum_{(\lambda,\lambda') \in (v,\alpha)\Lambda(n)^m} s_{(\lambda,\lambda')}s^*_{(\lambda,\lambda')} \\
    &= \sum_{\alpha \in v\Lambda^{<n}} s_{(v,\alpha)}
    = S_v.
\end{align*}
So $\{ S_\lambda \}_{ \lambda \in \Lambda }$ is a Cuntz-Krieger $\Lambda$-family. The
universal property of $C^*(\Lambda)$ now gives a homomorphism
$\ftn{\tilde{\iota}_{n}}{C^*(\Lambda)}{C^*(\Lambda(n))}$ such that
$\tilde{\iota}_{n}(s_\lambda) = S_\lambda = q_n(\iota_{n}(t_\lambda))$. The quotient maps
$\ftn{q}{\Tt C^*(\Lambda)}{C^*(\Lambda)}$ and $\ftn{q_n}{\Tt
C^*(\Lambda(n))}{C^*(\Lambda(n))}$ satisfy $\tilde{\iota}_{n} \circ q = q_n \circ
\iota_{n}$, and the formula for $\tilde{\iota}_n$ in~\eqref{eq:iota formula} follows.
\end{proof}

Now we construct an analogue of the map of
\cite[Proposition~3.1]{RuizSimsTomforde:arXiv:1312.0507}. For $\lambda \in \Lambda$, we
write $T(\lambda)$ for the unique path such that $\lambda = [\lambda]T(\lambda)$. Note
that $d( T( \lambda ) ) = d(\lambda) - [d(\lambda)] \in H_{n}$.

For a set $X$, we write $\Kk_X$ for the $C^*$-algebra of compact operators on
$\ell^2(X)$, with canonical matrix units $\setof{\theta_{x, y}}{x,y \in X}$.

\begin{lem}\label{lem:return homomorphism}
Let $\{s_{\lambda} \}_{\lambda \in \Lambda} \subseteq C^*(\Lambda)$ and $\{s_{(\lambda,
\lambda')} \}_{(\lambda, \lambda') \in \Lambda(n)} \subseteq C^*(\Lambda(n))$ be the
generating Cuntz-Krieger families. There is a homomorphism
$\ftn{j_n}{C^*(\Lambda(n))}{C^*(\Lambda) \otimes \Kk_{\Lambda^{<n}}}$ such that
\[j_n(s_{(\lambda,\lambda')})
    = s_{T(\lambda\lambda')} \otimes \theta_{[\lambda\lambda'], \lambda'}
        \quad\text{for all $(\lambda,\lambda') \in \Lambda(n)$}.
\]
\end{lem}
\begin{proof}
We just have to check the Cuntz-Krieger relations for the elements
$S_{(\lambda,\lambda')} := s_{T(\lambda\lambda')} \otimes \theta_{[\lambda\lambda'],
\lambda'}$. For $\lambda \in \Lambda^{<n}$, we have $T(\lambda) = s(\lambda)$ and
$[\lambda] = \lambda$. Thus $\{ S_{(r(\lambda),\lambda)} = s_{s(\lambda)}
\otimes \theta_{\lambda, \lambda} \}_{ ( r(\lambda), \lambda) \in \Lambda(n)^{0}}$ is a collection of mutually orthogonal projections.

Let $(\lambda, \lambda')$ and $( \mu, \mu' )$ be elements in $\Lambda(n)$.  Then
\begin{align*}S_{(\lambda,\lambda')}S_{(\mu,\mu')}
    &= (s_{T(\lambda\lambda')} \otimes \theta_{[\lambda\lambda'], \lambda'})
        (s_{T(\mu\mu')} \otimes \theta_{[\mu\mu'], \mu'}) \\
    &= \delta_{ \lambda', [ \mu \mu' ] } s_{T(\lambda\lambda')}s_{T(\mu\mu')} \otimes \theta_{[\lambda\lambda'], \mu'}\\
    &= \delta_{ s( (\lambda, \lambda') ), r( (\mu, \mu' )) } s_{T(\lambda\lambda')}s_{T(\mu\mu')} \otimes \theta_{[\lambda\lambda'], \mu'}.
\end{align*}
Suppose that $\lambda' = s((\lambda,\lambda')) = r((\mu,\mu'))  = [\mu\mu']$.  Then
$r(T(\mu\mu')) = s([\mu\mu']) = s(\lambda') = s(T(\lambda\lambda'))$. Moreover,
$\lambda\lambda'T(\mu\mu') = \lambda\mu\mu'$ because $[\mu\mu'] = \lambda'$. So
$T(\lambda\lambda'T(\mu\mu')) = T(\lambda\mu\mu') = T(\lambda\lambda')T(\mu\mu')$. Since
$d(T(\mu\mu')) \in H$, we also have $[\lambda\lambda'] = [\lambda\lambda'T(\mu\mu')]$,
and hence $[\lambda\mu\mu'] = [\lambda\lambda'T(\mu\mu')] = [\lambda\lambda']$. Putting
these two observations together, we deduce that
\begin{align*}
S_{(\lambda,\lambda')}S_{(\mu,\mu')}
    &= \delta_{ s( (\lambda, \lambda') ), r( (\mu, \mu' )) } s_{T(\lambda\mu\mu')} \otimes \theta_{[\lambda\mu\mu'], \mu'} \\
    &= \delta_{ s( (\lambda, \lambda') ), r( (\mu, \mu' )) } S_{(\lambda\mu,\mu')} \\
    &= \delta_{ s( (\lambda, \lambda') ), r( (\mu, \mu' )) } S_{(\lambda,\lambda')(\mu,\mu')},
\end{align*}
establishing~(TCK2).

For~(TCK3), fix $(\lambda,\lambda') \in \Lambda(n)$. We have
\begin{align*}
S^*_{(\lambda,\lambda')} S_{(\lambda,\lambda')}
    &= \big(s^*_{T(\lambda\lambda')} \otimes \theta_{\lambda', [\lambda\lambda']}\big)
            \big(s_{T(\lambda\lambda')} \otimes \theta_{[\lambda\lambda'], \lambda'}\big)\\
    &= s_{s(T(\lambda\lambda'))} \otimes \theta_{\lambda', \lambda'}
    = s_{s(\lambda')} \otimes \theta_{\lambda', \lambda'}
    = S_{(r(\lambda'),\lambda')}
    = S_{s((\lambda,\lambda'))}.
\end{align*}

Finally for~(CK), fix $(v,\lambda) \in \Lambda(n)^0$ and $m \in \NN^k$. Then
\begin{align*}
\sum_{(\mu,\mu') \in (v,\lambda)\Lambda(n)^m} S_{(\mu,\mu')} S^*_{(\mu,\mu')}
    &= \sum_{\mu \in v\Lambda^m, \mu' \in s(\mu)\Lambda^{<n}, [\mu\mu'] = \lambda}
        S_{(\mu,\mu')} S^*_{(\mu,\mu')} \\
    &= \sum_{\mu \in v\Lambda^m, \mu' \in s(\mu)\Lambda^{<n}, [\mu\mu'] = \lambda}
        s_{T(\mu\mu')} s^*_{T(\mu\mu')} \otimes \theta_{[\mu\mu'], [\mu\mu']}\\
    &= \sum_{\mu \in v\Lambda^m, \mu' \in s(\mu)\Lambda^{<n}, [\mu\mu'] = \lambda}
        s_{T(\mu\mu')} s^*_{T(\mu\mu')} \otimes \theta_{\lambda, \lambda}.
\end{align*}
Let $p := m + [d(\lambda) - m]$. Then $p \ge 0$ and $[p] = d(\lambda)$, so $p \ge
d(\lambda)$. The factorisation property implies that $\setof{\mu\mu'}{\mu \in v\Lambda^m,
\mu' \in s(\mu)\Lambda^{<n}, [\mu\mu'] = \lambda } = \setof{\lambda\nu}{ \nu \in
s(\lambda)\Lambda^{p - d(\lambda)} }$. For $\nu \in s(\lambda)\Lambda^{p - d(\lambda)}$,
we have $T(\lambda\nu) = \nu$. We deduce that
\[\sum_{(\mu,\mu') \in (v,\lambda)\Lambda(n)^m} S_{(\mu,\mu')} S^*_{(\mu,\mu')}
    = \sum_{\nu \in s(\lambda)\Lambda^{p - d(\lambda)}} s_\nu s^*_\nu \otimes \theta_{\lambda,\lambda}
    = s_{s(\lambda)} \otimes \theta_{\lambda,\lambda}
    = S_{(v,\lambda)}
\]
as required. Now the universal property of $C^*(\Lambda(n))$ gives the desired
homomorphism $j_n$.
\end{proof}

\section{Cartesian products, $1$-graphs, and the Kribs-Solel construction}\label{sec:CP1-graphKS}

Kumjian and Pask show that a cartesian product $\Lambda \times \Gamma$ of higher-rank
graphs is itself a higher-rank graph with $C^*(\Lambda \times \Gamma) \cong C^*(\Lambda)
\otimes C^*(\Gamma)$ (\cite[Corollary~3.5(iv)]{KumjianPask:NYJM00}). In this section we
show that the construction of the preceding section is compatible with the
cartesian-product operation, and also that the construction of \cite{KribsSolel:JAMS07}
and that of the preceding section are compatible via the passage from directed graphs to
$1$-graphs. We will use these results to compute the map $\ftn{K_*(j_n \circ
\tilde\iota_n)}{K_*(C^*(\Lambda))}{K_*(C^*(\Lambda) \otimes \Kk_{\Lambda^{<n}} )} \cong
K_*(C^*(\Lambda))$ for a particular class of $2$-graphs $\Lambda$.

For $i = 1,2$, let $(\Lambda_{i}, d_{i})$ be a $k_{i}$-graph.  The product category
$(\Lambda_{1} \times \Lambda_{2}, d_{1} \times d_{2})$ is a $(k_{1} + k_{2})$-graph with
degree map $(d_{1} \times d_{2})((\mu_{1}, \mu_{2})) = (d_{1} (\mu_{1}), d (\mu_{2}))$.
If each $(\Lambda_{i}, d_{i})$ is row-finite with no sources, then so is $(\Lambda_{1}
\times \Lambda_{2}, d_{1} \times d_{2})$. By
\cite[Corollary~3.5(iv)]{KumjianPask:NYJM00}, there exists an isomorphism
$\ftn{\Theta_{\Lambda_{1} \times \Lambda_{2}}}{C^{*} (\Lambda_{1} \times
\Lambda_{2})}{C^{*} (\Lambda_{1}) \otimes C^{*} (\Lambda_{2})}$ such that
$\Theta_{\Lambda_{1} \times \Lambda_{2}} (s_{(\mu_{1}, \mu_{2})}) = s_{\mu_{1}} \otimes
s_{\mu_{2}}$.

\begin{rmk}\label{rmk:pathcat}
Let $(\Lambda_{i}, d_{i})$ be a row-finite $k_{i}$-graph with no sources for $i = 1,2$.
For $n_1 \in \NN^{k_1}$ and $n_2 \in \NN^{k_2}$, the functor that sends $((\lambda_{1},
\lambda_{2}), (\lambda_{1}', \lambda_{2}')) \in (\Lambda_{1} \times \Lambda_{2})((n_{1},
n_{2}))$ to $((\lambda_{1}, \lambda_{1}'), (\lambda_{2}, \lambda_{2}')) \in  \Lambda_{1}
(n_{1}) \times \Lambda_{2} (n_{2})$ is an isomorphism of $(k_{1} + k_{2})$-graphs. So
there is an isomorphism $C^{*} ((\Lambda_{1} \times \Lambda_{2})((n_{1}, n_{2})) ) \cong
C^{*}(\Lambda_{1} (n_{1}) \times \Lambda_{2} (n_{2}) )$ sending $s_{((\lambda_{1},
\lambda_{2}), (\lambda_{1}', \lambda_{2}'))}$ to $s_{((\lambda_{1}, \lambda_{1}'),
(\lambda_{2}, \lambda_{2}'))}$.
\end{rmk}

We show that the homomorphism in Lemma~\ref{l:homomorphism KribSolel} is compatible with
the isomorphism $C^{*}((\Lambda_{1} \times \Lambda_{2})((n_{1}, n_{2})) ) \cong
C^{*}(\Lambda_{1} (n_{1}) \times \Lambda_{2} (n_{2}))$ just described.

\begin{lem}\label{l:prodgraph}
For $i = 1,2$, let $(\Lambda_{i}, d_{i})$ be a row-finite $k_{i}$-graph with no sources.
For $(n_{1}, n_{2})  \in\NN^{k_{1}} \times \NN^{k_{2}}$, we have $(\tilde{\iota}_{n_{1}} \otimes \tilde{\iota}_{n_{2}}) \circ \Theta_{\Lambda_{1} \times \Lambda_{2}}
    = (\Theta_{\Lambda_{1}(n_{1}) \times \Lambda_{2}(n_{2})}) \circ \tilde{\iota}_{(n_{1}, n_{2})}$.
\end{lem}
\begin{proof}
Let $(\mu_{1}, \mu_{2}) \in  \Lambda_{1} \times \Lambda_{2}$.  Then
\[(\tilde{\iota}_{n_{1}} \otimes \tilde{\iota}_{n_{2}}) \circ \Theta_{\Lambda_{1} \times \Lambda_{2}} (s_{(\mu_{1}, \mu_{2})})
    = \tilde{\iota}_{n_{1}} \otimes \tilde{\iota}_{n_{2}}  (s_{\mu_{1}} \otimes s_{\mu_{2}})
    = \sum_{\substack{\nu \in s(\mu_{1}) \Lambda_{1}^{<n_{1}}  \\ \nu' \in s(\mu_{2}) \Lambda_{2}^{<n_{2}}}} s_{(\mu_{1}, \nu)} \otimes s_{(\mu_{2}, \nu')}.
\]
Identifying $((\Lambda_{1} \times \Lambda_{2})((n_{1}, n_{2})), d_{1} \times d_{2})$ with
$(\Lambda_{1} (n_{1}), d_{1}) \times (\Lambda_{2} (n_{2}) , d_{2})$ as in
Remark~\ref{rmk:pathcat}, we have
\begin{align*}
\Theta_{\Lambda_{1}(n_{1}) \times \Lambda_{2}(n_{2})} \circ \tilde{\iota}_{(n_{1}, n_{2})} (s_{(\mu_{1}, \mu_{2})})
    &= \Theta_{\Lambda_{1}(n_{1}) \times \Lambda_{2}(n_{2})}
         \Big(\sum_{(\alpha, \beta) \in s((\mu_{1}, \mu_{2})) (\Lambda_{1} \times \Lambda_{2})^{< (n_{1}, n_{2})}}
                    s_{((\mu_{1}, \alpha), (\mu_{2}, \beta))}
         \Big) \\
	&= \sum_{\substack{\alpha \in s(\mu_{1}) \Lambda^{<n_{1}} \\ \beta \in s(\mu_{2}) \Lambda^{< n_{2}}}}
            \Theta_{\Lambda_{1}(n_{1}) \times \Lambda_{2}(n_{2})} (s_{((\mu_{1}, \alpha), (\mu_{2}, \beta))}) \\
	&= \sum_{\substack{\alpha \in s(\mu_{1}) \Lambda^{<n_{1}} \\ \beta \in s(\mu_{2}) \Lambda^{< n_{2}}}}
            s_{(\mu_{1}, \alpha)} \otimes s_{(\mu_{2}, \beta)}.
\end{align*}
Therefore, $\tilde{\iota}_{n_{1}} \otimes \tilde{\iota}_{n_{2}} \circ \Theta_{\Lambda_{1}
\times \Lambda_{2}} = \Theta_{\Lambda_{1}(n_{1}) \times \Lambda_{2}(n_{2})} \circ
\tilde{\iota}_{(n_{1}, n_{2})}$.
\end{proof}

We will need to apply Lemma~\ref{l:prodgraph} where $\Lambda_1$ and $\Lambda_2$ are the
$1$-graphs associated to directed graphs $E$ and $F$, and relate this to
\cite[Lemma~2.5]{RuizSimsTomforde:arXiv:1312.0507} for $C^*(E)$ and $C^*(F)$. We
therefore find ourselves in an unfortunate clash of conventions. The convention used in
\cite{RuizSimsTomforde:arXiv:1312.0507} is that of \cite{KPRR, KPR} where,
for historical reasons, the partial
isometries in a Cuntz-Krieger family point in the opposite direction to the edges in the
graph. This is at odds with the $k$-graph convention where the partial isometries go in
the same direction as the morphisms in the $k$-graph. To deal with this, we take the
approach that the range and source maps are interchanged when passing from a directed
graph $E$ to its path category $E^*$.

We recall the definition of the Toeplitz algebra $\Tt C^*(E)$ and the Cuntz-Krieger
algebra $C^*(E)$ of a directed graph $E$ as used in
\cite{RuizSimsTomforde:arXiv:1312.0507}. Let $E = (E^{0}, E^{1}, r_{E}, s_{E})$ be a
row-finite directed graph with no sinks (so $0 < |\setof{e \in E^{1}}{s_{E} (e) = v}| <
\infty$  for $v \in E^{0}$). Then $\Tt C^*(E)$ is the universal $C^*$-algebra generated
by mutually orthogonal projections $\{q_v\}_{v \in E^0 }$ and elements $\{t_e\}_{ e
\in E^1 }$ such that
\begin{itemize}
\item[1)] $t^*_e t_e = q_{r_{E}(e)}$ for all $e \in E^1$, and
\item[2)] $q_v \ge \sum_{e \in E^1, s_{E}(e) = v} t_e t^*_e$ for each $v \in E^0$.
\end{itemize}
The graph $C^{*}$-algebra $C^*(E)$ is the universal $C^*$-algebra generated by mutually
orthogonal projections $\{p_v\}_{v \in E^0 }$ and elements $\{s_e\}_{ e \in E^1 }$
such that
\begin{itemize}
\item[3)] $s^*_e s_e = p_{r_{E}(e)}$ for all $e \in E^1$, and
\item[4)] $p_v = \sum_{e \in E^1, s_{E}(e)= v} t_e t^*_e$ for each $v \in E^0$.
\end{itemize}

We recall the construction described in \cite[Section~4]{KribsSolel:JAMS07}.  Given $m
\in \NN$ and a directed graph $E = (E^{0}, E^{1}, r_{E}, s_{E})$, we define $E(m)$ to be
the directed graph with
\begin{align*}
E(m)^0 &= E^{<m},&
E(m)^1 &= \setof{(e, \mu)}{ e \in E^1, \mu \in E^{<m}, r_{E}(e) = s_{E}(\mu) },\\
r_{E(m)}((e,\mu)) &= \mu,&
s_{E(m)}((e,\mu)) &= \begin{cases}
		e\mu &\text{if $|\mu| < m-1$}\\
		s_{E}(e) &\text{if $|\mu| = m-1$.}
	\end{cases}
\end{align*}

The next lemma is due to James Rout, and will appear in his PhD thesis. We thank James
for providing us with the details (a proof appears in
\cite[Lemma~2.5]{RuizSimsTomforde:arXiv:1312.0507}).

\begin{lem}[Rout]\label{l:Rout}
Let $E$ be a row-finite directed graph and take $m \ge 1$.  There is an injective homomorphism $\ftn{\iota_{m, E}}{\Tt
C^*(E)}{\Tt C^*(E(m))}$ such that
\[\iota_{m, E}(q_v) = \sum_{\substack{\mu \in E^{<m} \\ s_{E} (\mu) = v}} q^m_\mu\quad\text{and}\quad
\iota_{m, E}(t_e) = \sum_{(e,\mu)  \in E(m)^{1}} t^m_{(e,\mu)},
\]
where $\{q^{m}_{\mu} , t_{(e, \nu)}^{m} \}_{ \mu \in E(m)^{0} , (e, \nu ) \in E(m)^{1}}$
are the universal generators of $\Tt C^{*} (E)$.  The map $\iota_{m, E}$ descends to an
injective homomorphism $\ftn{\tilde{\iota}_{m, E}}{C^*(E)}{C^*(E(m))}$.
\end{lem}

We describe canonical isomorphisms $C^{*}(E) \cong C^{*} (E^{*})$ and $C^{*} (E(m)) \cong
C^{*} (E^{*} (m))$ and show that these isomorphisms intertwine the homomorphism
$\tilde{\iota}_{m, E}$ of Lemma~\ref{l:Rout} and the homomorphism $\tilde{\iota}_{m}$ of
Lemma~\ref{l:homomorphism KribSolel}.

\begin{rmk}\label{rmk:C*(E)congC*(E*)}
Let $E$ be a row-finite directed graph with no sinks, and let $E^*$ be its path-category regarded
as a row-finite $1$-graph with no sources. Let $\{p_v, s_e\}_{v \in E^0, e \in E^1}$ be
the universal generators of $C^*(E)$ and let $\{S_\lambda \}_{ \lambda \in E^* }$ be the
universal generators of $C^*(E^*)$. By~\cite[Examples~1.7]{KumjianPask:NYJM00}, there is
an isomorphism $\ftn{\psi_E}{C^*(E)}{C^*(E^*)}$ such that $\psi_E(p_v) = S_v$ and
$\psi_E(s_e) = S_e$ for all $v \in E^0$ and $e \in E^1$.
\end{rmk}

\begin{lem}\label{lem:E*(m)congE(m)*}
Let $E$ be a row-finite directed graph with no sinks, and let $E^*$ be its path-category
regarded as a row-finite $1$-graph with no sources. There is an isomorphism of $1$-graphs
$E(m)^* \cong E^*(m)$ extending the identity map on $(E(m)^*)^1 = E^*(m)^1$. There is an
isomorphism $C^*(E(m)^*) \cong C^*(E^*(m))$ satisfying $s_{(e,\mu)} \mapsto s_{(e,\mu)}$ for
$(e,\mu) \in E^*(m)^1 = (E(m)^*)^1$.
\end{lem}
\begin{proof}
Example~1.3 of \cite{KumjianPask:NYJM00} says that $1$-graphs $\Lambda$ and $\Gamma$ are
isomorphic if and only if there is a bijection $\Lambda^1 \to \Gamma^1$ that intertwines
range maps and source maps. Since $(e,\mu) \mapsto ( e, \mu )$ is such a bijection between $(E(m)^*)^1$ and
$E^*(m)^1$, there is an isomorphism $E^*(m) \cong E(m)^*$ as claimed. Since isomorphic
1-graphs have canonically isomorphic $C^*$-algebras, the result follows.
\end{proof}

\begin{lem}\label{l:equivalence construction}
Let $E$ be a row-finite directed graph with no sinks, and fix $m \in \NN \setminus\{0\}$.
Identify $C^*(E(m)^*)$ with $C^*(E^*(m))$ using Lemma~\ref{lem:E*(m)congE(m)*}. Then the
isomorphisms $\ftn{\psi_{E}}{C^{*} (E)}{C^{*} (E^{*})}$ and $\ftn{\psi_{E(m)}}{C^{*}
(E(m))}{C^{*} (E^{*}(m))}$ of Remark~\ref{rmk:C*(E)congC*(E*)} satisfy $\tilde{\iota}_{m}
\circ \psi_{E} = \psi_{E(m)} \circ \tilde{\iota}_{m,E}$.
\end{lem}
\begin{proof}
Let $v \in E^{0}$.  Then
\[\tilde{\iota}_{m} \circ \psi_{E}  (p_{v})
    = \tilde{\iota}_{m} (S_{v})
    = \sum_{\lambda \in v (E^{*})^{< m}} S_{(v, \lambda)}
    = \sum_{\substack{\lambda \in E^{< m} \\  s_{E}(\lambda) = v}} S_{(v, \lambda)},
\]
and
\[\psi_{E(m)} \circ \tilde{\iota}_{m,E} (p_{v})
    = \psi_{E(m)} \Big(\sum_{\substack{\lambda \in E^{<m} \\ s_{E} (\lambda) = v}} p_{\lambda} \Big)
    = \sum_{\substack{\lambda \in E^{<m} \\ s_{E} (\lambda) = v}} S_{(s_{E} (\lambda), \lambda)}
    = \sum_{\substack{\lambda \in E^{<m} \\ s_{E} (\lambda) = v}} S_{(v, \lambda)}.
\]
Thus, $\tilde{\iota}_{m} \circ \psi_{E}  (p_{v}) = \psi_{E(m)} \circ \tilde{\iota}_{m,E}
(p_{v})$ for all $v \in E^{0}$. For $e \in E^{1}$,
\[
\tilde{\iota}_{m} \circ \psi_{E}  (s_{e})
    = \tilde{\iota}_{m} (S_{e})
    = \sum_{\lambda \in s(e) (E^{*})^{<m}} S_{(e, \lambda)}
    = \sum_{\substack{\lambda \in E^{<m}  \\ s_{E} (\lambda)  = r_{E} (e)}} S_{(e, \lambda)}
    = \sum_{(e, \lambda) \in E(m)^{1}} S_{(e, \lambda)},
\]
and
\[\psi_{E(m)} \circ \tilde{\iota}_{m,E} (s_{e})
    = \psi_{E(m)} \Big(\sum_{(e, \lambda) \in E(m)^{1}} s_{(e, \lambda)} \Big)
    = \sum_{(e, \lambda) \in E(m)^{1}} S_{(e, \lambda)}.
\]
So $\tilde{\iota}_{m} \circ \psi_{E}  (s_{e}) = \psi_{E(m)} \circ \tilde{\iota}_{m,E}
(s_{e})$ for all $e \in E^{1}$.  Since $C^{*} (E)$ is generated by $\{p_{v}, s_{e} \}_{ v
\in E^{0} , e \in E^{1}}$, we see that $\tilde{\iota}_{m} \circ \psi_{E}  = \psi_{E(m)}
\circ \tilde{\iota}_{m,E}$.
\end{proof}

\section{Asymptotic order-1 approximations}\label{sec:order1factorisations}

In this section, we show that given a row-finite $2$-graph with no sources, the family of
homomorphisms $(\tilde{\iota}_{n})_{n \in \NN^k}$ has an asymptotic order-1 approximation
through AF-algebras.  Thus, the family $( j_{n} \circ \tilde{\iota}_{n} )_{n \in
\NN^{k}}$ has an asymptotic order-1 approximation through AF-algebras.  We will use this
family of homomorphisms in the next section to prove that the nuclear dimension of a
UCT-Kirchberg algebra with trivial $K_{0}$ and finite $K_{1}$ is~1.

If $\ftn{f}{\NN^k}{\RR}$ is a function, then we write $\lim_{n \to \infty} f(n) = 0$ if
for every $\varepsilon > 0$ there exists $N \in \NN^k$ such that $|f(n)| < \varepsilon$
whenever $n \ge N$ in $\NN^k$.

Recall that a completely positive map $\ftn{\phi}{A}{B}$ has \emph{order-zero} if for $a,
b \in A_{+}$ with $ab = 0$, we have $\phi(a)\phi(b) = 0$. Suppose that $(\beta_n)_{n\in
\NN^k}$ is a family of homomorphisms $\ftn{\beta_n}{A}{B_n}$, and let $\mathcal{C}$ be a
class of $C^*$-algebras. Following
\cite[Definition~2.8]{RuizSimsTomforde:arXiv:1312.0507}\footnote{In the preprint version
of \cite{RuizSimsTomforde:arXiv:1312.0507} the authors mistakenly require just that each
$F_n$, rather than each $F_n^{(i)}$, belonged to $\mathcal{C}$; the intention was that
$\mathcal{C}$ should be closed under hereditary subalgebras and direct sums.}, a family
$(F_n, \phi_n, \psi_n)_{n\in \NN^k}$ is an \emph{asymptotic order-$r$ factorisation of
the family $(\beta_n)$ through elements of $\mathcal{C}$} if each $F_n$ is a direct sum
$F_n = \bigoplus^r_{i=0} F_n^{(i)}$ of $C^*$-algebras $F_n^{(i)} \in \mathcal{C}$, each
$\ftn{\psi_n}{A}{F_n}$ is a completely positive contraction, each
$\ftn{\phi_n}{F_n}{B_n}$ restricts to an order-zero completely positive contraction on
each $F_n^{(i)}$, and $\lim_{n \to \infty} \|\phi_n \circ \psi_n(a) - \beta_n(a)\| = 0$
for each $a \in A$. We say that $(F_n, \phi_n, \psi_n)_{n \in \NN^k}$ is an
\emph{asymptotic order-$r$ factorisation of $\ftn{\beta}{A}{B}$} if it is an asymptotic
order-$r$ factorisation of $(\beta)_{n \in \NN^k}$.

\begin{rmk}\label{rmk:approximation subseq}
Suppose that $(\ftn{\beta_n}{A}{B_n})_{n \in \NN^k}$ has an asymptotic order-$r$
factorisation through elements of $\mathcal{C}$. Then for any strictly increasing
sequence $(n^m)_{m \in \NN}$ in $\NN^k$ such that $n^m_j \to \infty$ as $m \to \infty$
for each $j \le k$, the sequence $(\beta_{n^m})_{m \in \NN}$ has an asymptotic order-$r$
factorisation through elements of $\mathcal{C}$ in the sense of
\cite[Definition~2.8]{RuizSimsTomforde:arXiv:1312.0507}.
\end{rmk}

Throughout this section, we use the following notation.  Let $\Lambda$ be a row-finite
$k$-graph with no sources and let $n \in \NN^{k}$. Then $\{t_{\lambda} \}_{\lambda \in
\Lambda} \subseteq \Tt C^{*} (\Lambda)$ and $\{T_{(\lambda,\lambda')}
\}_{(\lambda,\lambda') \in \Lambda(n)} \subseteq \Tt C^{*} (\Lambda(n))$ will be the
universal generating Toeplitz-Cuntz-Krieger families, and $\{s_{\lambda} \}_{\lambda \in
\Lambda } \subseteq C^{*} (\Lambda)$ and $\{S_{(\lambda, \lambda')} \}_{(\lambda,
\lambda') \in \Lambda(n)} \subseteq C^{*} (\Lambda(n))$ will be the universal generating
Cuntz-Krieger families. We will regard $\Tt C^{*} (\Lambda)$ as a sub-$C^{*}$-algebra of
$B(\ell^{2} (\Lambda))$. When $s(\mu) = s(\nu)$, we have
\begin{equation}\label{eq:Toeplitz series}
    t_{\mu} t_{\nu}^{*} = \sum_{\tau \in s(\mu)\Lambda} \theta_{\mu \tau, \nu \tau},
\end{equation}
where the series converges in the strict topology.

First we construct a homomorphism that we will use to define the maps $\phi_n$ in our
asymptotic factorisation.

\begin{lem}\label{lem:Lambda homomorphism}
Let $\Lambda$ be a row-finite $k$-graph with no sources. For $p,n \in \NN^{k}$, there is
a homomorphism $\ftn{\Gamma_p^{p+n}}{\bigoplus_{v \in \Lambda^0} \Kk_{\Lambda^{[p, p+n)}
v}}{\Tt C^*(\Lambda(n))}$ such that
\[\Gamma_p^{p+n}(\theta_{\mu,\nu}) = T_{(\mu, s(\mu))} T^*_{(\nu, s(\nu))}
\]
for all $\mu,\nu \in \Lambda^{[p, p+n)}$ with $s(\mu) = s(\nu)$.
\end{lem}
\begin{proof}
We just have to check that the elements $\{T_{(\mu, s(\mu))} T^*_{(\nu, s(\nu))}\}_{\mu,
\nu \in \Lambda^{[p, p+n)} , s(\mu) = s(\nu) }$ are nonzero and are matrix units in the
sense that $(T_{(\mu, s(\mu))} T^*_{(\nu, s(\nu))})^* = T_{(\nu, s(\nu))} T^*_{(\mu,
s(\mu))}$ and
\[
T_{(\mu, s(\mu))} T^*_{(\nu, s(\nu))}T_{(\mu', s(\mu'))} T^*_{(\nu', s(\nu'))}
    = \delta_{\nu,\mu'} T_{(\mu, s(\mu))} T^*_{(\nu', s(\nu'))}.
\]
(It follows from the displayed equation that $\lsp\{(T_{(\mu, s(\mu))} T^*_{(\nu,
s(\nu))})^* : s(\mu) = s(\nu) = v\} \perp \lsp\{(T_{(\mu, s(\mu))} T^*_{(\nu, s(\nu))})^*
: s(\mu) = s(\nu) = w\}$ for distinct $v,w$.)

The $T_{(\mu, s(\mu))} T^*_{(\nu, s(\nu))}$ are nonzero by~\eqref{eq:Toeplitz series}.
Let $\mu, \nu \in \Lambda^{[p, p+n)}$. By Lemma~\ref{lem:Lambda(n)min},
\begin{align*}
&\Lambda(n)^{\mathrm{min}} ((\nu, s(\nu)), (\mu, s(\mu))) \\
    &\ = \begin{cases}
    \setof{((\alpha, \tau), (\beta, \tau))}
             {(\alpha, \beta) \in \Lambda^{\mathrm{min}} (\nu, \mu), \tau \in s(\alpha) \Lambda^{<n},
                [\alpha \tau] = s(\nu), [\beta \tau] = s(\mu)} &\text{ if $[ \mu ] = [ \nu ]$} \\
               \emptyset &\text{ otherwise.}
               \end{cases}
\end{align*}
We claim that
\[
\Lambda(n)^{\mathrm{min}} ((\nu, s(\nu)), (\mu, s(\mu)))
    = \begin{cases}
        \big\{\big((s(\nu), s (\nu)), (s(\nu ), s (\nu))\big)\big\} &\text{ if $\mu = \nu$}\\
        \emptyset &\text{ otherwise.}
    \end{cases}
\]
Indeed, if $\Lambda(n)^{\mathrm{min}} ((\nu, s(\nu)), (\mu, s(\mu))) \neq \emptyset$, say
$((\alpha, \tau), (\beta, \tau)) \in \Lambda(n)^{\mathrm{min}}((\nu, s(\nu)), (\mu,
s(\mu)))$, then $[ \mu ] = [ \nu ]$.  In particular, $[ d( \mu ) ] = d( [ \mu ] ) = d( [ \nu ] ) = [ d( \nu ) ]$.  Since $p \leq d(\nu), d(\mu) < p + n$, we have that $d(\nu) = d(\mu)$. Since
$\mu\alpha = \nu\beta$, the factorisation property forces $\mu = \nu$. We then have
$\Lambda^{\mathrm{min}} (\nu, \mu) = \Lambda^{\mathrm{min}} (\nu, \nu) = \{(s(\nu),
s(\nu)) \}$, giving
\[
\Lambda(n)^{\mathrm{min}} ((\nu, s(\nu)), (\mu, s(\mu))) = \{((s(\nu), s(\nu)), (s(\nu), s(\nu))) \}
\]
as claimed.

We now show that $\{ T_{(\mu, s(\mu))}T^*_{(\nu, s(\nu))} \}_{\mu, \nu \in \Lambda^{[p, p+n)} , s( \mu
) = s(\nu) }$ form a system of matrix units, so that
the formula given for $\Gamma_p^{p+n}$ indeed defines a homomorphism. For $\mu, \nu,
\mu', \nu' \in \Lambda^{[p, p+n)}$ with $s(\mu) = s(\nu)$ and $s (\mu') = s(\nu')$,
\begin{align*}
T_{(\mu, s(\mu))}T^*_{(\nu, s(\nu))}&  T_{(\mu', s(\mu'))} T_{(\nu', s(\nu'))}^{*} \\
    &= T_{(\mu, s(\mu))} \Big(\sum_{((\alpha, \gamma), (\beta, \delta))  \in  \Lambda(n)^{\mathrm{min}} ((\nu, s(\nu)), (\mu', s(\mu')))}
                                        T_{(\alpha, \gamma)} T_{(\beta, \delta)}^{*} \Big) T_{(\nu', s(\nu'))}^{*} \\
    &= \delta_{\nu, \mu'} T_{(\mu, s(\mu))} T_{(s(\nu), s(\nu))} T_{(\nu', s(\nu'))}^{*} \\
    &= \delta_{\nu, \mu'} T_{(\mu, s(\mu))} T_{(s(\mu), s(\mu))} T_{(s(\nu'), s(\nu'))} T_{(\nu', s(\nu'))}^{*} \\
    &= \delta_{\nu, \mu'} \delta_{s(\mu), s(\nu')} T_{(\mu, s(\mu))}T_{(\nu', s(\nu'))}^{*}.\qedhere
\end{align*}
\end{proof}

Next we provide a technical lemma and a proposition that summarises what we require to construct an
approximate order-1 factorisation of the family $(\tilde\iota_n)_{n \in \NN^k}$ obtained
from Lemma~\ref{l:homomorphism KribSolel}.

\begin{lem}\label{lem:equiv-exprs}
Let $\Lambda$ be a row-finite $k$-graph with no sources and let $n \in \NN^k$.  For each
$\mu \in \Lambda$
\begin{equation}\label{eq:iota_n relation}
    \iota_{n} (t_{\mu}) T_{(s(\mu), s(\mu))} = T_{(\mu, s(\mu))} =  T_{(r(\mu), [\mu])} \iota_{n} (t_{\mu}).
\end{equation}
For $\mu, \nu, \tau \in \Lambda$ with $s(\mu) = s(\nu) = r(\tau)$,
\[
T_{(\mu \tau, s(\mu \tau))} T_{(\nu \tau, s(\nu \tau))}^{*} = \iota_{n} (t_{\mu}) \iota_{n} (t_{\tau} t_{\tau}^{*}) T_{(r(\tau), [\tau])} \iota_{n} (t_\nu^{*})
    \quad \text{and} \quad T_{(\mu, s(\mu))} T_{(\mu, s(\mu))}^{*} = \iota_{n}(t_{\mu} t_{\mu}^{*}) T_{(r(\mu), [\mu])}.
\]
\end{lem}
\begin{proof}
Recall that $\iota_{n} (t_{\mu}) = \sum_{\lambda \in s(\mu) \Lambda^{<n}} T_{(\mu,
\lambda)}$.  So
\begin{align*}
\iota_{n} (t_{\mu}) T_{(s(\mu), s(\mu))}
    &=  \Big(\sum_{\lambda \in s(\mu) \Lambda^{<n}} T_{(\mu, \lambda)} \Big) T_{(s(\mu), s(\mu))}\\
    &=  \Big(\sum_{\lambda \in s(\mu) \Lambda^{<n}} T_{(\mu, \lambda)} T_{(s(\mu), \lambda)} \Big) T_{(s(\mu), s(\mu))}
	= T_{(\mu, s(\mu))}.
\end{align*}

We now prove that $ T_{(\mu, s(\mu))} = T_{(r(\mu), [\mu])} \iota_{n} (t_{\mu} )$. We
have
\[T_{(r(\mu), [\mu])} \iota_{n} (t_{\mu})
    = T_{(r(\mu), [\mu])} \sum_{\lambda \in s(\mu) \Lambda^{< n}} T_{(\mu, \lambda)}
    =  T_{(r(\mu), [\mu])} \sum_{\lambda \in s(\mu) \Lambda^{< n}} T_{(r(\mu), [\mu \lambda])} T_{(\mu, \lambda)}.
\]
Note that $T_{(r(\mu), [\mu])} T_{(r(\mu), [\mu \lambda])} \neq 0$ if and only if $[\mu]
= [\mu \lambda]$.  Let $\lambda \in s(\mu) \Lambda^{<n}$ with $[\mu] = [\mu \lambda]$.
Since $[\mu] = \mu (0,  [d(\mu)])$ and $[\mu \lambda] = (\mu \lambda)(0, [d(\mu
\lambda)])$, we see that $d(\lambda) = d(\mu\lambda) - d(\mu) \in H_n$. Since $d(\lambda)
< n$, we deduce that $d(\lambda) = 0$, giving $\lambda = r(\lambda) = s(\mu)$. Hence,
\[T_{(r(\mu), [\mu])} \iota_{n} (t_{\mu})
    =  T_{(r(\mu), [\mu])} \sum_{\lambda \in s(\mu) \Lambda^{< n}} T_{(r(\mu), [\mu \lambda])} T_{(\mu, \lambda)}
    = T_{(\mu, s(\mu))}.
\]
This proves~\eqref{eq:iota_n relation}.

For the second assertion, take $\mu, \nu, \tau \in \Lambda$ with $s(\mu) = s(\nu) = r
(\tau)$. Then~\eqref{eq:iota_n relation} gives
\begin{align*}
T_{(\mu \tau, s(\mu \tau))} T_{(\nu \tau, s(\nu \tau))}^{*}
    &= \iota_{n}(t_{\mu\tau}) T_{(s(\mu \tau), s(\mu \tau))} T_{(s(\nu \tau), s(\nu \tau))} \iota_{n} (t_{\nu \tau}^{*})
	= \iota_{n} (t_{\mu}) \iota_{n} (t_{\tau}) T_{(s(\tau), s(\tau))} \iota_{n} (t_{\tau}^{*}) \iota_{n} (t_{\nu}^{*}) \\
	&= \iota_{n} (t_{\mu}) \iota_{n} (t_{\tau}) \iota_{n} (t_{\tau}^{*}) T_{(r(\tau), [\tau])} \iota_{n} (t_{\nu}^{*})
	= \iota_{n} (t_{\mu}) \iota_{n} (t_{\tau} t_{\tau}^{*}) T_{(r(\tau), [\tau])} \iota_{n} (t_{\nu}^{*}),
\end{align*}
and
\begin{align*}
T_{(\mu, s(\mu))} T_{(\mu, s(\mu))}^{*}
	&= \iota_{n} (t_{\mu}) T_{(s (\mu), s(\mu))} \big(\iota_{n} (t_{\mu})  T_{(s (\mu), s(\mu))} \big)^{*}\\
    &= \iota_{n} (t_{\mu}) T_{(s (\mu), s(\mu))} \iota_{n} (t_{\mu}^{*})
	= \iota_{n}(t_{\mu} t_{\mu}^{*}) T_{(r(\mu), [\mu])}.\qedhere
\end{align*}
\end{proof}

Recall that for $n \in \NN^k$ with each $n_i
\ge 1$, the group $H_n$ is the subgroup
\[
	\setof{p \in \ZZ^k}{ n_i \text{ divides } p_i\text{ for each }i \le k}.
\]
For $x \in \RR^{k}$, let $\lceil x \rceil = \big(\lceil x_{1} \rceil, \dots, \lceil x_{k}
\rceil  \big) \in \ZZ^k$, and for $a \in \RR \setminus\{0\}$, put $\frac{x}{a} =
\big(\frac{{x}_{1}}{a}, \dots, \frac{x_{k}}{a} \big)$.

\begin{prp}\label{p: order 1 approximation}
Let $\Lambda$ be a row-finite $k$-graph with no sources. For each $n \in \NN^k$ such that
each $n_j > 0$, each $p < n$, and each $\mu \in \Lambda$, let $h_{n,\mu}(p)$ and
$g_{n,\mu}(p)$ be the unique elements in $H_n$ such that
\[\textstyle
n \leq d(\mu) + p + h_{n, \mu}(p) < 2n
    \quad \text{and} \quad \Big \lceil \frac{3n}{2} \Big \rceil \leq d(\mu) + p + g_{n, \mu}(p) <  \Big \lceil \frac{5n}{2} \Big \rceil.
\]
For each $n \in \NN^k$, let $\Delta_n$ be a function $\ftn{\Delta_n}{\NN^{k} \times
\NN^{k}}{[0,1)}$. For $i = 1,2$, define $\Delta_{n,i}^{\mu, \nu} : \NN^k  \to [0,1)$ by
$\Delta_{n,1}^{\mu, \nu}(p) := \Delta_{n} \big(d(\mu) + p + h_{n, \mu}(p) - n , d (\nu) +
p + h_{n,\mu}(p) - n\big)$ and $\Delta_{n,2}^{\mu, \nu}(p) := \Delta_{n} \big(d(\mu) + p
+ g_{n, \mu}(p) - \lceil \frac{3n}{2} \rceil, d (\nu) + p + g_{n, \mu}(p) - \lceil
\frac{3n}{2} \rceil\big)$. Suppose that for each $\mu, \nu \in \Lambda$,
\[
\lim_{n \to \infty} \max \big\{|\Delta_{n,1}^{\mu, \nu}(p) + \Delta_{n,2}^{\mu, \nu}(p) - 1 | \mathbin{\big|} p < n\big\} = 0.
\]
Suppose that there exist completely positive, contractive linear maps $\ftn{P_n}{\Tt C^{*}
(\Lambda)}{\mathcal{K}_{\Lambda^{[n, 2n)}}}$ and $\ftn{Q_n}{\Tt C^{*}
(\Lambda)}{\mathcal{K}_{\Lambda^{[\lceil 3n/2 \rceil, \lceil 5n/2 \rceil)}}}$ such that
\[
P_n(t_{\mu} t_{\nu}^{*})
    = \sum_{\substack{\tau \in s(\mu) \Lambda  \\ n \leq d(\mu \tau), d(\nu \tau)  < 2n}}
        \Delta_{n} (d(\mu \tau) - n, d (\nu \tau) - n) \theta_{\mu\tau, \nu \tau}
\]
and
\[
Q_n (t_{\mu} t_{\nu}^{*})
    = \sum_{\substack{\tau \in s(\mu) \Lambda  \\
            \left \lceil \frac{3n}{2} \right \rceil
                \leq d(\mu \tau), d(\nu \tau)
                < \left \lceil \frac{5n}{2} \right \rceil}}
                    {\textstyle\Delta_{n} \left(d(\mu \tau) - \left \lceil \frac{3n}{2} \right \rceil, d (\nu \tau)
                        - \left \lceil \frac{3n}{2} \right \rceil  \right) \theta_{\mu\tau, \nu \tau}}
\]
for all $\mu, \nu \in \Lambda$ with $s(\mu) = s(\nu)$. Then the family
$(\tilde{\iota}_{n})_{n \in \NN^k}$ has an order-1 approximation through AF-algebras.
\end{prp}
\begin{proof}
For each $n \in \NN^k$, let $\ftn{\pi_n}{\Tt C^{*} (\Lambda(n))}{C^{*} (\Lambda(n))}$ be
the quotient homomorphism.  We first show that for all $\mu, \nu \in \Lambda$ with
$s(\mu) = s(\nu)$,
\begin{equation}\label{eq:approx iota}
\lim_{n \to \infty} \Big\| \pi_{n}\Big(\Big(\big(\Gamma_{n}^{2n} \circ P_n
                                + \Gamma_{\lceil \frac{3n}{2} \rceil}^{\lceil \frac{5n}{2} \rceil} \circ Q_n\big) - \iota_{n}\Big)
                                    (t_{\mu} t_{\nu}^{*})\Big)\Big\| = 0,
\end{equation}
where the $\Gamma$'s are the homomorphisms constructed in Lemma~\ref{lem:Lambda
homomorphism}. For this, let $\mu, \nu \in \Lambda$ with $s(\mu) = s(\nu)$, and fix $n
\in \NN^k$. Lemma~\ref{lem:Lambda homomorphism} gives
\[
\Gamma_{n}^{2n} \circ P_{n} (t_{\mu} t_{\nu}^{*})
    = \sum_{\substack{\tau \in s(\mu) \Lambda  \\ n \leq d(\mu \tau), d(\nu \tau)  < 2n}}
        \Delta_{n} (d(\mu \tau) - n, d (\nu \tau) - n) T_{(\mu \tau, s(\mu \tau))} T_{(\nu \tau, s(\nu \tau))}^{*}.
\]
Lemma~\ref{lem:equiv-exprs} shows that $T_{(\mu \tau, s(\mu \tau))} T_{(\nu \tau, s(\nu
\tau))}^{*} = \iota_{n} (t_{\mu}) \iota_{n} (t_{\tau} t_{\tau}^{*}) T_{(r(\tau), [\tau])}
\iota_{n} (t_\nu^{*})$. So,
\begin{align*}
\Gamma_{n}^{2n} \circ P_{n} (t_{\mu} t_{\nu}^{*})
    &= \iota_{n} (t_{\mu})  \Big(\sum_{\substack{\tau \in s(\mu) \Lambda  \\ n \leq d(\mu \tau), d(\nu \tau)  < 2n}}
        \Delta_{n} (d(\mu \tau) - n, d (\nu \tau) - n) \iota_{n} (t_{\tau} t_{\tau}^{*}) T_{(r(\tau), [\tau])} \Big) \iota_{n} (t_\nu^{*})  \\
    &=\iota_{n} (t_{\mu}) \Big(\sum_{p < n} \sum_{\alpha \in s(\mu) \Lambda^{p}} \sum_{\rho \in s(\alpha) \Lambda^{h_{n,\mu}(p)}}
        \Delta_{n,1}^{\mu, \nu}(p)  \iota_{n} (t_{\alpha \rho} t_{\alpha \rho}^{*}) T_{(r(\alpha \rho), [\alpha \rho])} \Big) \iota_{n} (t_\nu^{*}) \\
    &= \iota_{n} (t_{\mu}) \Big(\sum_{p < n} \sum_{\alpha \in s(\mu) \Lambda^{p}} \sum_{\rho \in s(\alpha) \Lambda^{h_{n,\mu}(p)}}
        \Delta_{n,1}^{\mu, \nu}(p)  \iota_{n} (t_{\alpha}) \iota_{n}(t_{\rho} t_{\rho}^{*}) \iota_{n} (t_{\alpha}^{*})
            T_{(s(\mu), \alpha)} \Big) \iota_{n} (t_\nu^{*}) \\
    &=  \iota_{n} (t_{\mu}) \Big(\sum_{p < n} \sum_{\alpha \in s(\mu) \Lambda^{p}}
        \Delta_{n,1}^{\mu, \nu}(p)  \iota_{n} (t_{\alpha})
            \Big(\sum_{\rho \in s(\alpha) \Lambda^{h_{n,\mu}(p)}} \iota_{n}(t_{\rho} t_{\rho}^{*}) \Big)
            \iota_{n} (t_{\alpha}^{*})  T_{(s(\mu), \alpha)} \Big) \iota_{n} (t_\nu^{*}).
\end{align*}
Relation~(CK) for $\{s_\lambda\}_{\lambda \in \Lambda}$ gives
\begin{align*}
\pi_{n} &{}\circ \Gamma_{n}^{2n} \circ P_n (t_{\mu} t_{\nu}^{*}) \\
    &= \widetilde{\iota}_{n} (s_{\mu}) \Big(\sum_{p < n} \sum_{\alpha \in s(\mu) \Lambda^{p}}
        \Delta_{n,1}^{\mu, \nu}(p)  \widetilde{\iota}_{n} (s_{\alpha}) \widetilde{\iota}_{n} (s_{s(\alpha)}) \widetilde{\iota}_{n} (s_{\alpha}^{*})
            S_{(s(\mu), \alpha)} \Big) \widetilde{\iota}_{n} (s_\nu^{*}) \\
    &= \widetilde{\iota}_{n} (s_{\mu}) \Big(\sum_{p < n} \sum_{\alpha \in s(\mu) \Lambda^{p}}
        \Delta_{n,1}^{\mu, \nu}(p) \widetilde{\iota}_{n} (s_{\alpha}) \widetilde{\iota}_{n} (s_{\alpha}^{*})
            S_{(s(\mu), \alpha)} \Big) \widetilde{\iota}_{n} (s_\nu^{*})
\end{align*}
By Lemma~\ref{lem:equiv-exprs}, $S_{(\alpha, s(\alpha))} S_{(\alpha, s(\alpha))}^{*} =
\widetilde{\iota}_{n}(s_{\alpha} s_{\alpha}^{*}) S_{(r(\alpha), [\alpha])} =
\widetilde{\iota}_{n} (s_{\alpha}) \widetilde{\iota}_{n} (s_{\alpha}^{*}) S_{(s(\mu),
\alpha)}$ for all $\alpha \in s(\mu) \Lambda^{<n}$, and hence
\begin{equation}\label{eq:piGammaP sum}
\pi_n \circ \Gamma^{2n}_n \circ P_n(t_\mu t^*_\nu)
    = \widetilde{\iota}_{n} (s_{\mu}) \Big(\sum_{p < n} \sum_{\alpha \in s(\mu) \Lambda^{p}}
        \Delta_{n,1}^{\mu, \nu}(p) S_{(\alpha, s(\alpha))} S_{(\alpha, s(\alpha))}^{*} \Big) \widetilde{\iota}_{n} (s_\nu^{*}).
\end{equation}

Take $p < n$ and $\alpha \in s(\mu) \Lambda^{p}$. Then
\begin{align*}
\{(\lambda, \lambda') \in \Lambda(n)^{p}  \mid  {}&r((\lambda, \lambda')) = (s(\mu), \alpha)\} \\
    &= \setof{(\lambda, \lambda') \in \Lambda \times \Lambda^{< n}}{s(\lambda) = r(\lambda'), d(\lambda) = p,
        (r(\lambda), [\lambda \lambda']) = (s(\mu), \alpha)}.
\end{align*}
Suppose that $(\lambda, \lambda') \in \Lambda(n)^{p}$ with $r ((\lambda, \lambda')) =
(s(\mu), \alpha)$.  Then $[\lambda\lambda'] = \alpha$, so $p = d(\alpha) = d([\lambda
\lambda']) = [d(\lambda \lambda')] = [p + d(\lambda')]$. Hence $[p] = [p + d(\lambda')]$,
and since $d(\lambda') < n$, this forces $d(\lambda') = 0$.  Therefore, $\alpha = [ \lambda \lambda' ] = [ \lambda ] = \lambda$ since $d( \lambda ) = p < n$ and $\lambda' = r( \lambda' ) = s( \lambda ) = s( \alpha )$.  Hence
\[
\setof{(\lambda, \lambda') \in \Lambda(n)^{p}}{ r((\lambda, \lambda')) = (s(\mu), \alpha)} = \{(\alpha, s(\alpha))\},
\]
which implies that each $S_{(\alpha, s(\alpha))} S_{(\alpha, s(\alpha))}^{*} =
S_{(s(\mu), \alpha)}$ by~(CK) in $C^*(\Lambda(n))$. Combining this
with~\eqref{eq:piGammaP sum} gives
\[\pi_{n} \circ \Gamma_{n}^{2n} \circ P_n (t_{\mu} t_{\nu}^{*})
    = \widetilde{\iota}_{n} (s_{\mu}) \Big(\sum_{p < n} \sum_{\alpha \in s(\mu) \Lambda^{p}}
        \Delta_{n,1}^{\mu, \nu}(p) S_{(s(\mu), \alpha)}  \Big) \widetilde{\iota}_{n} (s_\nu^{*}).
\]
A similar computation gives
\[\pi_{n} \circ \Gamma_{\lceil \frac{3n}{2}  \rceil}^{\lceil \frac{5n}{2} \rceil} \circ Q_n (t_{\mu} t_{\nu}^{*})
    =  \widetilde{\iota}_{n} (s_{\mu}) \Big(\sum_{p < n} \sum_{\alpha \in s(\mu) \Lambda^{p}}
        \Delta_{n,2}^{\mu, \nu}(p)  S_{(s(\mu), \alpha)}  \Big) \widetilde{\iota}_{n} (s_\nu^{*}).
\]

Since $\{S_{(s(\mu), \alpha)} \}_{ \alpha \in s(\mu) \Lambda^{<n}}$ is a collection of mutually orthogonal projections,
\begin{align*}
\big\| \big(\pi_{n} \circ \Gamma_{n}^{2n}& \circ P_n (t_{\mu} t_{\nu}^{*})
            + \pi_{n} \circ \Gamma_{\lceil \frac{3n}{2} \rceil}^{\lceil \frac{5n}{2} \rceil} \circ Q_n (t_{\mu} t_{\nu}^{*})\big)
            - \pi_{n} \circ \iota_{n} (t_{\mu} t_{\nu}^{*}) \big\| \\
    &\leq \Big\|\sum_{p < n} \sum_{\alpha \in s(\mu) \Lambda^{p}}
            (\Delta_{n,1}^{\mu, \nu}(p)  + \Delta_{n,2}^{\mu, \nu}(p)) S_{(s(\mu), \alpha)}
                - \widetilde{\iota}_{n} (s_{s(\mu)}) \Big\| \\
    &= \Big\|\sum_{p < n} \sum_{\alpha \in s(\mu) \Lambda^{p}}
            (\Delta_{n,1}^{\mu, \nu}(p)  + \Delta_{n,2}^{\mu, \nu}(p)) S_{(s(\mu), \alpha)}
            - \sum_{p < n} \sum_{\alpha \in s(\mu) \Lambda^{p}} S_{(s(\mu), \alpha)} \Big\| \\
    &= \max_{p < n} | \Delta_{n,1}^{\mu, \nu}(p)  + \Delta_{n,2}^{\mu, \nu}(p)  - 1|.
\end{align*}
By assumption, $\lim_{n \to \infty} \max_{p < n} | \Delta_{n,1}^{\mu, \nu}(p) +
\Delta_{n,2}^{\mu, \nu}(p)  - 1 | = 0$.  This proves~\eqref{eq:approx iota}.

Since $k$-graph algebras are nuclear \cite[Theorem~5.5]{KumjianPask:NYJM00}, we may apply
\cite[Theorem~3.10]{ChoiEffros:ANMATH} to obtain a contractive completely positive
splitting  $\ftn{\sigma}{C^{*} (\Lambda)}{\Tt C^{*} (\Lambda)}$ for the quotient map. For
each $n$, define $\ftn{\psi_{n}}{C^{*}(\Lambda)}{\Kk_{\Lambda^{[n, 2n)}} \oplus
\Kk_{\Lambda^{[\lceil 3n/2 \rceil, \lceil 5n/2 \rceil)}}}$ by $\psi_n(a) :=
\big(P_n(\sigma(a)), Q_n(\sigma(a))\big)$ and $\ftn{\phi_{n}}{\Kk_{\Lambda^{[n, 2n)}}
\oplus \Kk_{\Lambda^{[\lceil 3n/2 \rceil, \lceil 5n/2 \rceil)}}}{C^{*}(\Lambda(n))}$ by
$\phi_{n} ( ( a, b ) ) = \pi_{n} (\Gamma_{n}^{2n}(a) +\Gamma^{\lceil \frac{5n}{2}
\rceil}_{\lceil \frac{3n}{2} \rceil}(b))$. By Lemma~\ref{lem:Lambda homomorphism},
$\phi_{n}$ restricts to a homomorphism (and in particular an order-zero map) on each of
$\Kk_{\Lambda^{[n, 2n)}}$ and $\Kk_{\Lambda^{[\lceil 3n/2 \rceil, \lceil 5n/2 \rceil)}}$.
Since $\Tt C^{*} ( \Lambda ) = \overline{ \mathrm{span} } \setof{ t_{\mu} t_{\nu}^{*} }{
\mu, \nu \in \Lambda, s( \mu ) = s( \nu )}$ and since
\[\lim_{n \to \infty} \big\| \pi_{n} \circ \Gamma_{n}^{2n} \circ P_n (t_{\mu} t_{\nu}^{*})
        + \pi_{n} \circ \Gamma_{\lceil \frac{3n}{2} \rceil}^{\lceil \frac{5n}{2} \rceil} \circ Q_n (t_{\mu} t_{\nu}^{*})
        -  \pi_{n} \circ \iota_{n} (t_{\mu} t_{\nu}^{*}) \big\| = 0
\]
for all $\mu, \nu \in \Lambda$ with $s(\mu) = s(\nu)$, the family $\left(\Kk_{\Lambda^{[n, 2n)}} \oplus \Kk_{\Lambda^{[\lceil 3n/2 \rceil,
\lceil 5n/2 \rceil)}}, \psi_{n}, \phi_{n} \right)$ is an asymptotic order-1 approximation
of $(\tilde{\iota}_{n})_{n \in \NN^k}$ through AF-algebras.
\end{proof}

\begin{ntn}
Following \cite{WinterZacharias:Adv10}, for each $m \in \NN$, define $\kappa_{m} \in
M_{\{0, \dots, m-1\}}\big([0,1)\big)$ as follows: put $l := \lceil\frac{m}{2}\rceil$, and
define
\begin{gather*}
\kappa_{m} = \frac{1}{l+1}
    \left(
    \begin{matrix}
        1 & 1 & \dots & 1 & 1 &\dots & 1 & 1 \\
        1 & 2 & \dots & 2 & 2 &\dots & 2 & 1 \\
        \vdots & \vdots & \ddots  &\vdots &\vdots & \iddots & \vdots &\vdots  \\
        1 & 2 & \dots & l & l & \dots & 2 & 1 \\
        1 & 2 & \dots & l & l & \dots & 2 & 1 \\
        \vdots & \vdots & \iddots &\vdots & \vdots & \ddots &\vdots &\vdots  \\
        1 & 2 & \dots & 2 & 2 &\dots & 2 & 1 \\
        1 & 1 & \dots & 1 & 1 &\dots & 1 & 1 \\
    \end{matrix}
    \right) \qquad\text{ if $m$ is even} \\
\kappa_{m} = \frac{1}{l+2}
    \left(
    \begin{matrix}
        1 & 1 & \dots & 1 &\dots & 1 & 1 \\
        1 & 2 & \dots & 2 &\dots & 2 & 1 \\
        \vdots & \vdots & \ddots &\vdots & \iddots & \vdots &\vdots  \\
        1 & 2 & \dots & l+1 & \dots & 2 & 1 \\
        \vdots & \vdots & \iddots &\vdots & \ddots &\vdots &\vdots  \\
        1 & 2 & \dots & 2 &\dots & 2 & 1 \\
        1 & 1 & \dots & 1 &\dots & 1 & 1 \\
    \end{matrix}
    \right) \qquad\text{ if $m$ is odd.}
\end{gather*}
Define $\kappa_{m}(i,j) = 0$ for $(i,j) \in \ZZ^2 \setminus (\{0, \dots, m-1\} \times
\{0, \dots, m-1\})$.
\end{ntn}

\begin{thm}\label{thm: order 1 approximation}
Let $\Lambda$ be a row-finite $2$-graph with no sources. Then $(\tilde{\iota}_{n})_{n \in
\NN^k}$ has an asymptotic order-1 approximation through AF-algebras.
\end{thm}
\begin{proof}
For $m \in \NN$, let $\mathsf{A}_{m}$ denote the $m\times m$ matrix with all entries
equal to 1. For $n \in \NN^2$, define
\[\Delta_{n}
    := \frac{1}{2} \left(\kappa_{n_1} \otimes \mathsf{A}_{n_2} + \mathsf{A}_{n_1} \otimes \kappa_{n_2}\right)
    : \CC^{n_1} \otimes \CC^{n_2} \to \CC^{n_1} \otimes \CC^{n_2}.
\]
Since $\CC^{n_1} \otimes \CC^{n_2}$ is a finite-dimensional Hilbert space, $\Delta_{n}$
can be regarded as an $n_1n_2 \times n_1n_2$ matrix.  Since $\kappa_{n_1} \otimes
\mathsf{A}_{n_2}$ and $\mathsf{A}_{n_1} \otimes \kappa_{n_2}$ are positive elements in
the $C^{*}$-algebra $M_{n_1} \otimes M_{n_2}$, the matrix $\Delta_{n}$ is also positive.
Write $\{e_{i}\}$ for the canonical orthonormal basis elements of $\CC^{n_1}$ and of
$\CC^{n_2}$. Then
\begin{align*}
\Delta_{n} (i_{1}, i_{2}, j_{1}, j_{2}) &= \langle \Delta_{n} e_{i_{1}} \otimes e_{i_{2}}, e_{j_{1}} \otimes e_{j_{2}} \rangle \\
	&= \frac{1}{2} \big(\langle \kappa_{n_1} e_{i_{1}},  e_{j_{1}} \rangle \langle \mathsf{A}_{n_2} e_{i_{2}}, e_{j_{2}} \rangle
            + \langle \mathsf{A}_{n_1} e_{i_{1}},  e_{j_{1}} \rangle \langle \kappa_{n_2} e_{i_{2}}, e_{j_{2}} \rangle \big) \\
	&= \frac{1}{2} \big(\kappa_{n_1} (i_{1}, j_{1}) + \kappa_{n_2} (i_{2}, j_{2})\big).
\end{align*}
Define $M^{n,1} \in M_{\Lambda^{[n, 2n)}}$ by $M^{n,1}_{\mu, \nu} = \Delta_{n} (d(\mu) -
n, d(\nu) - n)$ and define $M^{n,2} \in M_{\Lambda^{[\lceil3n/2\rceil,
\lceil5n/2\rceil)}}$ by $M^{n,2}_{\mu, \nu} = \Delta_{n} (d(\mu) - \lceil \frac{3n}{2}
\rceil,  d(\nu) -\lceil \frac{3n}{2} \rceil)$. We claim that Schur multiplication by
$M^{n,i}$ is a completely positive contraction for $i = 1,2$. We just argue the case $i =
1$ and when $n_1$ and $n_2$ are even; the other cases are similar. For $1 \le j \le
n_1/2$, let $\Phi^{1,j}$ be the strong-operator sum $\sum_{|d(\lambda)_1 - (3n_1-1)/2| <
j} \theta_{\lambda,\lambda}$, and for $1 \le j \le n_2/2$, let $\Phi^{2,j} =
\sum_{|d(\lambda)_2 - (3n_2-1)/2| < j} \theta_{\lambda,\lambda}$, where $d( \lambda )_{i}$ denotes the $i$th coordinate of $d( \lambda )$. Each $\Phi^{i,j}$ is a
projection, and so $\Phi^i : a \mapsto \sum^{n_i/2}_{j=1} \frac{1}{n_i/2 + 1} \Phi^{i,j}
a \Phi^{i,j}$ is a completely positive contraction. Schur multiplication by $M^{n,1}$ is
equal to $\frac12(\Phi^1 + \Phi^2)$, and so is itself a completely positive contraction.

For $p < q \in \NN^2$, define $R^{q}_p \in \Bb(\ell^2(\Lambda))$ to be the
strong-operator sum
\[\textstyle
R^q_p
    = \sum_{\lambda \in \Lambda^{[p,q)}} \theta_{\lambda, \lambda}.
\]
Define $\ftn{P_{n}}{\Tt C^{*} (\Lambda)}{\Kk_{\Lambda^{[n, 2n)}}} \subseteq \Kk_\Lambda$
by $P_{n}(a)  =  M^{n,1}
* (R^{2n}_n a R^{2n}_n)$ and define $\ftn{Q_{n}}{\Tt C^{*} (\Lambda)}{\Kk_{\Lambda^{[\lceil
3n/2 \rceil, \lceil 5n/2 \rceil)}}} \subseteq \Kk_\Lambda$ by $Q_{n}(a) = M^{n,2} *
(R^{\lceil 5n/2 \rceil}_{\lceil 3n/2 \rceil} a R^{\lceil 5n/2 \rceil}_{\lceil 3n/2
\rceil})$. Since Schur multiplication by each $M^{n,i}$ is a completely positive,
contractive linear map and since $R^{2n}_n$ and $R^{\lceil 5n/2 \rceil}_{\lceil 3n/2
\rceil}$ are projections, $P_{n}$ and $Q_{n}$ are completely positive, contractive linear
maps.

We will show that $\Delta_{n}$, $P_{n}$, and $Q_{n}$ satisfy the hypotheses of
Proposition~\ref{p: order 1 approximation}.  Fix $\mu, \nu \in \Lambda$ with $s(\mu) =
s(\nu)$.  Recall that the series $\sum_{\tau \in \Lambda} \theta_{\mu \tau, \nu \tau}$
converges strictly to $t_\mu t^*_\nu$. Since $\theta_{\lambda, \lambda} \theta_{\mu \tau,
\nu \tau} \theta_{\beta, \beta} = \delta_{\lambda, \mu\tau} \delta_{\beta,\nu\tau}
\theta_{\mu \tau, \nu \tau}$,
\[
R^{2n}_n t_{\mu} t_{\nu}^{*} R^{2n}_n
    = \sum_{\substack{\tau \in s(\mu)\Lambda \\ n \leq d(\mu \tau), d(\nu \tau) < 2n}} \theta_{\mu\tau, \nu\tau}  \quad \text{and} \quad
R^{\lceil 5n/2 \rceil}_{\lceil 3n/2\rceil} t_{\mu} t_{\nu}^{*} R^{\lceil 5n/2 \rceil}_{\lceil 3n/2\rceil}
    = \sum_{\substack{\tau \in s(\mu)\Lambda \\
             \left\lceil \frac{3n}{2} \right\rceil \leq d(\mu \tau), d(\nu \tau) < \left \lceil \frac{5n}{2} \right\rceil}}
        \theta_{\mu\tau, \nu \tau}.
\]
Hence,
\[P_{n} (t_{\mu} t_{\nu}^{*})
    = \sum_{\substack{\tau \in s(\mu) \Lambda  \\ n \leq d(\mu \tau), d(\nu \tau)  < 2n}}
        \Delta_{n} (d(\mu \tau) - n, d (\nu \tau) - n) \theta_{\mu\tau, \nu \tau}
\]
and
\begin{align*}
Q_{n}(t_{\mu} t_{\nu}^{*})
    &= \sum_{\substack{\tau \in s(\mu) \Lambda  \\
        \left \lceil \frac{3n}{2} \right \rceil \leq d(\mu \tau), d(\nu \tau)  < \left \lceil \frac{5n}{2} \right \rceil}}
            {\textstyle\Delta_{n} \left(d(\mu \tau) - \left \lceil \frac{3n}{2} \right \rceil, d (\nu \tau)  - \left \lceil \frac{3n}{2} \right \rceil\right)
                \theta_{\mu\tau, \nu \tau}}.
\end{align*}

Let $p = (p_{1}, p_{2}) < n$, let $d(\mu) = (a_{1}, a_{2})$, and let $d(\nu) = (b_{1},
b_{2})$.  Let $h_{n, \mu}(p) = (h_{p_{1},n}^{\mu}, h_{p_{2}, n}^{\mu})$ be the unique
element in $H_{n}$ such that $n \leq d(\mu) + p + h_{n, \mu} (p) < 2n$ and let $g_{n,
\mu}(p) = (g_{p_{1}, n}^{\mu}, g_{p_{2}, n}^{\mu})$ be the unique element in $H_{n}$ such
that $\left\lceil \frac{3n}{2} \right\rceil \leq d(\mu) + p + g_{n, \mu} (p) <
\left\lceil \frac{5n}{2} \right\rceil$. Note that $h_{p_{j}, n}^{\mu}$ is the unique
element in $n_j\ZZ$ such that $n_j \leq a_j + p_{j} + h_{p_{j}, n}^{\mu} < 2n_j$ and
$g_{p_{j}, n}^{\mu}$ is the unique element in $n_j\ZZ$ such that $\big\lceil
\frac{3n_j}{2} \big\rceil \leq a_j + p_{j} + g_{p_{j}, n}^{\mu} < \big\lceil
\frac{5n_j}{2} \big\rceil$.

Set
\begin{align*}
\zeta_{n, \mu, p_{j}} := \kappa_{n_j} &(a_{j} + p_{j} + h_{p_{j}, n}^{\mu} - n_j, b_{j} + p_{j} + h_{p_{j}, n}^{\mu} - n_j) \\
    &\textstyle{} + \kappa_{n_j} \left(a_{j} + p_{j} + g_{p_{j}, n}^{\mu} - \big\lceil \frac{3n_j}{2} \big\rceil,
                b_{j} + p_{j} + g_{p_{j}, n}^{\mu} - \big\lceil \frac{3n_j}{2} \big\rceil \right).
\end{align*}
Using the definitions of the $h^\mu_{p_j, n}$ and $g^\mu_{p_j, n}$, one checks that
\[\textstyle
\big(a_{j} + p_{j} + h_{p_{j}, n}^{\mu} - n_j\big)
    - \big(a_{j} + p_{j} + g_{p_{j}, n}^{\mu} - \big\lceil \frac{3n_j}{2} \big\rceil\big)
    \in \big\{\big\lceil \frac{n_j}{2} \big\rceil, -\big\lfloor \frac{n_j}{2} \big\rfloor\big\}.
\]
For any integer $k$ and for any $x,y$, we have $\kappa_{k}(x,y) \ge \kappa_k(x,x) -
|x-y|/(\lceil\frac{k}{2}\rceil + 1)$, and $(\lceil k/2\rceil + 1)/(\lceil k/2\rceil + 2)
\le \kappa_k(x,x) + \kappa_k(x+\lceil\frac{k}{2}\rceil, x + \lceil\frac{k}{2}\rceil) \le
1$. Hence
\[
\frac{\lceil \frac{k}{2}\rceil + 1}{\lceil \frac{k}{2}\rceil + 2} - \frac{2|x-y|}{\lceil\frac{k}{2}\rceil + 1}
    \le \kappa_k(x,y) + \kappa_k\left({\textstyle x + \big\lceil\frac{k}{2}\big\rceil, y + \big\lceil\frac{k}{2}\big\rceil}\right) \le 1.
\]
Applying this inequality with $k = n_j$, $x = \min\big\{a_{j} + p_{j} + h_{p_{j},
n}^{\mu} - n_j, a_{j} + p_{j} + g_{p_{j}, n}^{\mu} - \big\lceil
\frac{3n_j}{2}\big\rceil\big\}$ and $y := \min\big\{b_{j} + p_{j} + h_{p_{j}, n}^{\mu} -
n_j, b_{j} + p_{j} + g_{p_{j}, n}^{\mu} - \big\lceil \frac{3n_j}{2} \big\rceil\big\}$, we
see that $|\zeta_{n, \mu, p_{j}} - 1| \le \big(1 + 2|a_j -
b_j|\big)/\big(\big\lceil\frac{n_j}{2}\big\rceil + 1\big)$.

For $p < n$, set $\Delta_{n,1}^{\mu, \nu}(p)  =  \Delta_{n} (d(\mu) + p + h_{n, \mu}(p) -
n, d(\nu) + p + h_{n,\mu}(p) - n)$ and $\Delta_{n,2}^{\mu, \nu}(p)  = \Delta_{n}(d(\mu) +
p + g_{n, \mu}(p) - \lceil \frac{3n}{2} \rceil, d(\nu) + p + g_{n, \mu}(p) - \lceil
\frac{3n}{2} \rceil)$.  Then
\[
\left|\Delta_{n,1}^{\mu, \nu}(p) + \Delta_{n,2}^{\mu, \nu}(p)  - 1 \right|
    = \left| \frac{1}{2} \left(\zeta_{n, \mu, p_{1}} - 1 \right)  + \frac{1}{2} \left(\zeta_{n, \mu,  p_{2}} - 1 \right) \right|
	\leq \frac{2(1 + |a_{1} - b_{1}| + |a_{2} - b_{2}|)}{2(\min\{\lceil \frac{n_1}{2} \rceil, \lceil \frac{n_2}{2} \rceil \} + 1)}.
\]
Hence
\[
\lim_{n \to \infty} \max \big\{|\Delta_{n,1}^{\mu, \nu}(p) + \Delta_{n,2}^{\mu, \nu}(p) - 1| \mathbin{\big|} p < n\big\} = 0.
\]

So $\Delta_{n}$, $P_{n}$ and $Q_{n}$ satisfy the hypotheses of Proposition~\ref{p: order
1 approximation}, which then says that $(\tilde{\iota}_{n})_{n \in \NN^k}$ has an
asymptotic order-1 approximation through AF-algebras.
\end{proof}

\begin{cor}\label{cor: order 1 approximation}
If $E$ and $F$ are row-finite directed graphs with no sinks, then $(\tilde{\iota}_{m, E}
\otimes \tilde{\iota}_{m, F})_{m \in \NN}$ has an asymptotic order-1 approximation
through AF-algebras.
\end{cor}
\begin{proof}
Let $\ftn{\tilde{\iota}_{m,1}}{C^{*} (E^{*})}{C^{*} (E^{*} (m))}$,
$\ftn{\tilde{\iota}_{m,2}}{C^{*} (F^{*})}{C^{*} (F^{*} (m))}$, $\ftn{\tilde{\iota}_{(m,
m)}}{C^{*} (E^{*} \times F^{*})}{C^{*} ((E^{*} \times F^{*})((m,m)))}$ be the
homomorphisms defined in Lemma~\ref{l:homomorphism KribSolel} for the 1-graphs $E^{*}$,
$F^{*}$, and the 2-graph $E^{*} \times F^{*}$ respectively.  By Lemma~\ref{l:prodgraph},
$\tilde{\iota}_{m,1} \otimes \tilde{\iota }_{m, 2} = \Theta_{E^{*} (m) \times F^{*} (m)}
\circ  \tilde{\iota}_{(m, m)} \circ \Theta_{E^{*} \times F^{*}}^{-1}$, where
$\Theta_{E^{*} \times F^{*}}$ and $\Theta_{E^{*} (m) \times F^{*} (m)}$ are isomorphisms.
By Theorem~\ref{thm: order 1 approximation} and Remark~\ref{rmk:approximation subseq},
the sequence $(\tilde{\iota}_{(m, m)})_{m \in \NN}$ has an asymptotic order-1
approximation through AF-algebras.  Hence, $(\tilde{\iota}_{m,1} \otimes
\tilde{\iota}_{m, 2})_{ m \in \NN}$ has an asymptotic order-1 approximation through AF-algebras.  By
Lemma~\ref{l:equivalence construction}, there exist isomorphisms $\ftn{\psi_{E}}{C^{*}
(E)}{C^{*} (E^{*})}$, $\ftn{\psi_{F}}{C^{*} (F)}{C^{*} (F^{*})}$,
$\ftn{\psi_{E(m)}}{C^{*} (E(m))}{C^{*} (E^{*}(m))}$, and $\ftn{\psi_{F(m)}}{C^{*}
(F(m))}{C^{*} (F^{*}(m))}$ such that $\tilde{\iota}_{m,E} = \psi_{E(m)}^{-1} \circ
\tilde{\iota}_{m,1} \circ \psi_{E}$ and $\tilde{\iota}_{m,F} = \psi_{F(m)}^{-1} \circ
\tilde{\iota}_{m,2} \circ \psi_{F}$. Hence,
\[\tilde{\iota}_{m, E} \otimes \tilde{\iota}_{m, F} = (\psi_{E(m)} \otimes \psi_{F(m)})^{-1}
    \circ  (\tilde{\iota}_{m,1} \otimes  \tilde{\iota}_{m,2}) \circ (\psi_{E} \otimes \psi_{F}).
\]
Thus $(\tilde{\iota}_{m, E} \otimes \tilde{\iota}_{m, F})_{m \in \NN}$ has an asymptotic
order-1 approximation through AF-algebras.
\end{proof}

\section{Nuclear dimension of UCT-Kirchberg algebras}\label{sec:main}

In this section, we show that all UCT-Kirchberg algebras have nuclear dimension~1.  We
already know from \cite{Enders:arXiv:1405.6538} that every UCT-Kirchberg algebra with
torsion free $K_{1}$-group has nuclear dimension~1. So we first show that each
UCT-Kirchberg algebra with trivial $K_{0}$-group and finite $K_{1}$-group has nuclear
dimension~1, and then prove our main theorem.

\begin{dfn}
A \emph{Kirchberg algebra} is a separable, nuclear, simple, purely infinite
$C^{*}$-algebra.  A \emph{UCT-Kirchberg algebra} is a Kirchberg algebra in the UCT class
of \cite{RosenbergSchochet:Duke87}.
\end{dfn}

For each finite abelian group $T$, let $E_{T}$ be an infinite, row-finite, strongly
connected graph such that $K_*(C^{*} (E_{T})) = ( T, \{0\} )$ and $C^{*} (E_{T})$ is a UCT-Kirchberg algebra (note that strongly connected implies that $E_{T}$ has no sinks and
sources).  Let $F_{\ZZ}$ be an infinite, row-finite, strongly connected graph such that
$K_{*} (C^{*} (F_{\ZZ}))  = (\{0\}, \ZZ)$ and $C^{*} (F_{\ZZ})$ is a UCT-Kirchberg
algebra.  Note that $E_{T}$ and $F_{\ZZ}$ exist by
\cite[Theorem~1.2]{Szymanski:IUMJ2002}.

\begin{lem}\label{l:nuclear dimension 2-graph}
Let $T$ be a finite abelian group.  Then the nuclear dimension of $C^{*} (E_{T}) \otimes
C^{*} (F_{\ZZ})$ is 1.  Consequently, every UCT-Kirchberg algebra with $K_{0}$ trivial
and $K_{1}$ finite has nuclear dimension 1.
\end{lem}
\begin{proof}
Consider the directed graphs $E_T$ and $F_\ZZ$. For $k \in \NN$, let
\[
\ftn{ \tilde{\iota}_{k, E_{T}} }{C^{*} ( E_{T} ) }{ C^{*} (E_{T}(k) ) }
    \quad\text{ and }\quad
\ftn{\tilde{\iota }_{k, F_{\ZZ}}}{ C^{*} (F_{\ZZ} ) }{ C^{*} ( F_{\ZZ} (k) ) }
\]
be the homomorphisms given in Lemma~\ref{l:Rout} for $E_{T}$ and $F_{\ZZ}$ respectively.
Let
\[
\ftn{ j_{k, E_{T}} }{C^{*} ( E_{T}(k) ) }{ C^{*} (E_{T}) \otimes \Kk }
    \quad\text{ and }\quad
\ftn{j_{k, F_{\ZZ}}}{ C^{*} ( F_{\ZZ}(k) ) }{ C^{*} ( F_{\ZZ} ) \otimes \Kk}
\]
be the homomorphisms given in \cite[Proposition~3.1]{RuizSimsTomforde:arXiv:1312.0507}
for $E_{T}$ and $F_{\ZZ}$ respectively.

By Corollary~\ref{cor: order 1 approximation}, there is an asymptotic order-1
approximation through AF-algebras for $(\tilde{\iota}_{k, E_{T}} \otimes \tilde{\iota
}_{k, F_{\ZZ}})_{k \in \NN}$. The composition of this sequence of homomorphisms with
$j_{k, E_{T}} \otimes j_{k, F_{\ZZ}}$ gives an asymptotic order-1 approximation through
AF-algebras for $((j_{k, E_{T}} \circ \tilde{\iota}_{k, E_{T}}) \otimes (j_{k, F_{\ZZ}}
\circ \tilde{\iota}_{k, F_{\ZZ}}))_{k \in \NN}$.

For $m \in \NN$, let $p_{m} = (|T|+1)^m$. Since the order of each element of $T$ divides
$|T|$, multiplication by each $p_{m}$ induces the identity map on $T$. Set $\gamma_{m} =
(j_{p_{m}, E_{T}} \circ \tilde{\iota}_{p_{m}, E_{T}}) \otimes (j_{p_{m}, F_{\ZZ}} \circ
\tilde{\iota}_{p_{m}, F_{\ZZ}})$.  By construction, $(\gamma_{m})_{m \in \NN}$ has an
asymptotic order-1 approximation through AF-algebras.  The K\"unneth formula in
\cite{RosenbergSchochet:Duke87} combined with
\cite[Lemma~3.2]{RuizSimsTomforde:arXiv:1312.0507} shows that $K_{1} (\gamma_{m})$ is
multiplication by $p_{m}^{2}$ on $K_{1} (C^{*} (E_{T}) \otimes C^{*} (F_{\ZZ})) = T$.
Thus, $K_{1} (\gamma_{m}) = \id_T$. Since $K_{0} (C^{*} (E_{T}) \otimes C^{*} (F_{\ZZ}))
= 0$ the map $K_{0} (\gamma_{m})$ is trivially the identity. Since $E_{T}$ and $F_{\ZZ}$
are infinite directed graphs, $C^{*} (E_{T})$ and $C^{*} (F_{\ZZ})$ are non-unital UCT-Kirchberg
algebras, and hence stable. The Universal Coefficient Theorem in
\cite{RosenbergSchochet:Duke87} and the Kirchberg-Phillips classification (cf.
\cite{kirchpureinf} and \cite{Phillips:Doc00}), show that there exist an isomorphism
$\ftn{\beta_{m}}{(C^{*} (E_{T}) \otimes \Kk) \otimes (C^{*} (F_{\ZZ}) \otimes \Kk)}{C^{*}
(E_{T}) \otimes C^{*} (F_{\ZZ})}$ and a unitary $ u_{m}$ in $\mathcal{M}
( C^{*} (E_{T}) \otimes C^{*} (F_{\ZZ}) )$ for each $m \in \NN$ such that
\begin{equation}\label{eq:order-1 approximation id}
\lim_{m \to \infty} \| u_{m} (\beta_{m} \circ \gamma_{m})(a) u_{m}^{*} - a \| = 0
    \quad\text{ for all $a \in C^{*} (E_{T}) \otimes C^{*} (F_{\ZZ})$.}
\end{equation}
Since $(\gamma_{m})_{m \in \NN}$ has an asymptotic order-1 approximation through
AF-algebras, so does $(\mathrm{Ad}(u_{m}) \circ \beta_{m} \circ \gamma_{m})_{m \in \NN}$.
So~(\ref{eq:order-1 approximation id}) implies that $\mathrm{id}_{C^{*} (E_{T}) \otimes
C^{*} (F_{\ZZ})}$ has an asymptotic order-1 approximation through AF-algebras. Hence
\cite[Lemma~2.9]{RuizSimsTomforde:arXiv:1312.0507} shows that $\ndim\big(C^{*} (E_{T})
\otimes C^{*} (F_{\ZZ })\big) = 1$.

Let $A$ be a UCT-Kirchberg algebra with $K_{0}$ trivial and $K_{1}$ finite. Then $K_{*}
(A) \cong K_{*} (C^{*} (E_{T}) \otimes C^{*} (F_{\ZZ }) )$ where $T = K_{1} (A)$.  By
Kirchberg-Phillips classification, $A \otimes \Kk \cong C^{*} (E_{T}) \otimes C^{*}
(F_{\ZZ })$.  Hence, $\ndim(A \otimes \Kk) = 1$; so
\cite[Corollary~2.8]{WinterZacharias:Adv10} gives $\ndim(A) = 1$.
\end{proof}

\begin{rmk}\label{rmk:SpecSeq}
A key step in the preceding proof was to show that the maps $j_n \circ \iota_n :
C^*(\Lambda) \to C^*(\Lambda) \otimes \Kk_{\Lambda^{<n}}$ induce multiplication by
$n_1n_2$ in $K$-theory. We were able to do this using the K\"unneth formula because
tensor products of $1$-graph $C^*$-algebras provide a large enough class of models to
cover all the UCT-Kirchberg algebras in question. It seems likely that our techniques
could be used to compute the exact value of nuclear dimension for a large class of
nonsimple Kirchberg algebras with torsion in $K_1$ along the lines of
\cite{RuizSimsTomforde:arXiv:1312.0507}, but to do this, we need to know that $j_n \circ
\iota_n$ induces multiplication by $n_1 n_2$ in the $K$-theory of every ideal of
$C^*(\Lambda)$ for general $2$-graphs $\Lambda$. This can be proved using Evans'
calculation of $K$-theory for $k$-graph algebras \cite{Evans:NYJM08} and naturality of
Kasparov's spectral sequence (see \cite[page~185]{KumjianPaskSims:DM08}). But this would
require introducing extraneous notation or digging into the proofs in
\cite{Evans:NYJM08}, so we have not pursued this approach here.
\end{rmk}

\begin{dfn}
A homomorphism $\ftn{\phi} {A}{B}$ is called \emph{full} if for all $a \in A$ with $a
\neq 0$, the closed two-sided ideal generated by $\phi(a)$ is equal to $B$.
\end{dfn}

\begin{lem}\label{l:directlimitsimple}
Let $(\phi_n : A_n \to A_{n+1})^\infty_{n=1}$ be a directed system of $C^*$-algebras, and
set $A = \varinjlim (A_{n}, \phi_n)$.  If there exists $N \in \NN$ such that $\phi_n$ is
full for all $n \geq N$, then $A$ is simple.  If, in addition, each $A_{n}$ is a finite
direct sum of UCT-Kirchberg algebras, then $A$ is a UCT-Kirchberg algebra.
\end{lem}
\begin{proof}
Let $I$ be a nonzero ideal of $A$.  Then $I_{n} = \phi_{n, \infty}^{-1}(I)$ is an ideal
of $A_{n}$ and $\phi_n (I_{n}) \subseteq I_{n+1}$.  Since $I$ is nonzero, there exists
$M$ such that $I_{n} \neq 0$ for all $n \geq M$.  So for $n \geq \max\{N, M \}$ the ideal
generated by $\phi_n(I_{n})$ is $A_{n+1}$.  Since $\phi_n (I_{n}) \subseteq I_{n+1}$, we
have $I_{n+1} = A_{n+1}$ for $n \geq \max \{N, M \}$, and so $I = A$.

Suppose each $A_{n}$ is a finite direct sum of UCT-Kirchberg algebras.  Then every
nonzero projection of any $A_{n}$ is properly infinite.  Since $\phi_n$ is injective for
$n \geq N$ (because the maps are full), $\phi_n$ takes properly infinite projections to
properly infinite projections for all $n \geq N$.  Thus, every nonzero projection of $A$
is properly infinite.  Hence, $A$ is a purely infinite simple $C^{*}$-algebra.  Since
each $A_{n}$ is separable, nuclear, and in the UCT class, $A$ is too. Thus, $A$ is a
UCT-Kirchberg algebra.
\end{proof}

\begin{lem}\label{l:direct limit decomposition}
Let $A$ be a stable UCT-Kirchberg algebra. Then there exist sequences
$(\Lambda_n)_{n\in \NN}$ and $(\Gamma_n)_{n \in \NN}$ of row-finite $2$-graphs with no
sources, and homomorphisms $\phi_n : C^*(\Lambda_n) \oplus C^*(\Gamma_n) \to
C^*(\Lambda_{n+1}) \oplus C^*(\Gamma_{n+1})$ such that: each $C^*(\Lambda_n)$ and each
$C^*(\Gamma_n)$ is a stable UCT-Kirchberg algebra with finitely-generated $K$-theory;
each $K_1(C^*(\Lambda_n))$ is free abelian; each $K_0(C^*(\Gamma_n))$ is trivial and each
$K_1(C^*(\Gamma_n))$ is finite; and $A \cong \varinjlim(C^*(\Lambda_n) \oplus C^*(\Gamma_n),
\phi_n)$.
\end{lem}
\begin{proof}
Let $( G_{n,0} )_{n \in \NN}$ and $( G_{n,1} )_{ n \in \NN}$ be increasing families of finitely generated
abelian groups with $\bigcup_{n=1}^{\infty} G_{n, 0} = K_{0}(A)$ and $\bigcup_{n =
1}^{\infty} G_{n,1} = K_{1} (A)$.  Decompose each $G_{n,1}$ as $G_{n,1} = T( G_{n,1} )
\oplus F( G_{n,1} )$, where $T( G_{n,1} )$ is finite and $F( G_{n,1} )$ is free.

Let $E_{ \Kk }$ be the infinite row-finite graph $\cdots \to  \bullet \to \bullet \to \cdots$
(so $C^{*} ( E_{ \Kk} )  \cong \Kk$). For each $n$ apply
\cite[Theorem~1.2]{Szymanski:IUMJ2002} to obtain a row-finite strongly connected graph
$E_n$ such that $C^*(E_n)$ is a UCT Kirchberg algebra with $K_{*}(C^{*}( E_{n} ) ) =(
G_{n,0} , F( G_{n,1} ) )$.  Then each $\Lambda_n := E_{n}^{*} \times E_{\Kk}^{*}$ is a
row-finite 2-graph with no sources such that $C^*(\Lambda_n)$ is a stable UCT-Kirchberg
algebra with $K_{*} ( C^{*}(\Lambda_n)) = ( G_{n,0} , F( G_{n,1} ) )$. For each $n$, let
$E_{T(G_{n,1})}$ and $F_\ZZ$ be as in Lemma~\ref{l:nuclear dimension 2-graph} and the preceding
discussion. Let $\Gamma_n := E^*_{T(G_{n,1})} \times F^*_\ZZ$. Then $\Gamma_n$ is a
row-finite 2-graph with no sources, and $C^*(\Gamma_n)$ is a stable UCT-Kirchberg algebra with
$K_{*}(C^*(\Gamma_n)) = ( 0 , T( G_{n,1} ) )$.

Each $K_{*}(C^*(\Lambda_n) \oplus C^*(\Gamma_n)) = (G_{n,0} , G_{n,1} )$.  By
Kirchberg-Phillips (cf. \cite{kirchpureinf} and \cite{Phillips:Doc00}), for each $n \in
\NN$, there exists a full homomorphism $\ftn{\phi_n}{C^*(\Lambda_n) \oplus
C^*(\Gamma_n)}{C^*(\Lambda_{n+1}) \oplus C^*(\Gamma_{n+1})}$ which in $K$-theory induces
the inclusion map $G_{n,i} \hookrightarrow G_{n+1, i}$. Therefore, $K_{i} ( \varinjlim (
C^*(\Lambda_n) \oplus C^*(\Gamma_n) , \phi_{n} ) ) \cong K_{i} ( A )$.  So, by
Lemma~\ref{l:directlimitsimple} and the Kirchberg-Phillips classification, $A \cong
\varinjlim ( C^*(\Lambda_n) \oplus C^*(\Gamma_n), \phi_{n} )$.
\end{proof}

\begin{thm}
Every UCT-Kirchberg algebra has nuclear dimension~1.
\end{thm}
\begin{proof}
Let $A$ be a UCT-Kirchberg algebra. Since Kirchberg algebras are not AF,
\cite[Remarks~2.2(iii)]{WinterZacharias:Adv10} shows that $A$ has nuclear dimension at
least 1. Corollary~2.8 of \cite{WinterZacharias:Adv10} shows that the nuclear dimension
of $A \otimes \Kk$ is the same as that of $A$, so we may assume that $A$ is stable.

By Lemma~\ref{l:direct limit decomposition}, $A  \cong \varinjlim ( B_{n} \oplus C_{n} ,
\phi_n)$ where $B_{n}$ and $C_{n}$ are UCT-Kirchberg algebras such that $K_{1} ( B_{n} )$
is free, $K_{0} ( C_{n} ) = 0$, and $K_{1} ( C_{n} )$ is a finite abelian group. By
Lemma~\ref{l:nuclear dimension 2-graph}, $\ndim(C_{n}) = 1$.  By
\cite[Theorem~4.1]{Enders:arXiv:1405.6538}, $\ndim(B_{n}) = 1$. Proposition~2.3(i) of
\cite{WinterZacharias:Adv10} implies that each $B_{n} \oplus C_{n}$ has nuclear dimension
1.  It now follows from \cite[Proposition~2.3(iii)]{WinterZacharias:Adv10} that $A$ has
nuclear dimension 1.
\end{proof}

\end{document}